\documentclass{article}
\usepackage{amsfonts}
\usepackage{amssymb}
\usepackage{graphicx}
\usepackage{amsmath}
\usepackage{url}

\newtheorem{theorem}{Theorem}

\newtheorem{corollary}[theorem]{Corollary}

\newtheorem{lemma}[theorem]{Lemma}

\newtheorem{proposition}[theorem]{Proposition}

\newenvironment{proof}[1][Proof]{\textbf{#1.} }{\ \rule{0.5em}{0.5em}}
\input{tcilatex}
\begin{document}

\title{Generalized Stirling Numbers I}
\author{Claudio Pita-Ruiz \\
Universidad Panamericana.\\
Facultad de Ingenier\'{\i}a.\\
Augusto Rodin 498, M\'{e}xico,\\
Ciudad de Mexico, 03920,\\
M\'{e}xico.\\
email: cpita@up.edu.mx}
\date{}
\maketitle

\begin{abstract}
We consider generalized Stirling numbers of the second kind $%
S_{a,b,r}^{\alpha _{s},\beta _{s},r_{s},p_{s}}\left( p,k\right) $, $%
k=0,1,\ldots .rp+\sum_{s=2}^{L}r_{s}p_{s}$, where $a,b,\alpha _{s},\beta
_{s} $ are complex numbers, and $r,p,r_{s},p_{s}$ are non-negative integers
given, $s=2,\ldots ,L$. (The case $a=1,b=0,r=1,r_{s}p_{s}=0$, corresponds to
the standard Stirling numbers $S\left( p,k\right) $.) The numbers $%
S_{a,b,r}^{\alpha _{s},\beta _{s},r_{s},p_{s}}\left( p,k\right) $ are
connected with a generalization of Eulerian numbers and polynomials we
studied in previous works. This link allows us to propose (first, and then
to prove, specially in the case $r=r_{s}=1$) several results involving our
generalized Stirling numbers, including several families of new recurrences
for Stirling numbers of the second kind. In a future work we consider the
recurrence and the differential operator associated to the numbers $%
S_{a,b,r}^{\alpha _{s},\beta _{s},r_{s},p_{s}}\left( p,k\right) $.
\end{abstract}

\section{\label{Sec1}Introduction}

Throughout the work, $L$ will denote an arbitrary positive integer $\geq 2$
given, $r,r_{s}$ and $p,p_{s}$ will denote non-negative integers given, and
we will write $\sigma $ for the sum $\sum_{s=2}^{L}r_{s}p_{s}.$

Stirling numbers are nice mathematical objects studied along the years: they
contain important combinatorial information, and they have shown to be
connected with many other important mathematical and physical objects. (For
a comprehensive study of Stirling numbers, including connections with
physics, see \cite{M-S} and the hundreds of references therein.) Many
generalizations of Stirling numbers are known nowadays. In 1998 Hsu and
Shiue \cite{H-S1} presented a unified approach containing several important
generalizations of Stirling numbers studied before by different authors
(Carlitz \cite{Ca1,Ca2,Ca3}, Howard \cite{How}, Gould-Hopper \cite{GouHop},
Riordan \cite{Rio}, Charalambides \cite{Char1,Char-Kou,Char-Singh}, Koutras 
\cite{Kou}, among others). We also mention the work of P. Blasiak \cite{Bl1}%
. Both, the Hsu and Shiue, and Blasiak works have some (natural)
intersections with the generalization we consider in this work. However, all
these works run in different directions.

We will be dealing with the $Z$-Transform (see \cite{Gr,Vi}), which is a
function $\mathcal{Z}$ that maps complex sequences $a_{n}=\left(
a_{0},a_{1},\ldots \right) $ into complex functions $\mathcal{Z}\left(
a_{n}\right) \left( z\right) $ (or simply $\mathcal{Z}\left( a_{n}\right) $%
), given by the Laurent series $\mathcal{Z}\left( a_{n}\right) \left(
z\right) =\sum_{n=0}^{\infty }\frac{a_{n}}{z^{n}}$ (called $Z$-Transform of
the sequence $a_{n},$ defined in the exterior of the circle of convergence
of the Taylor series $\sum_{n=0}^{\infty }a_{n}z^{n}$ ---the generating
function of the sequence $a_{n}$---). If $\mathcal{Z}\left( a_{n}\right) =%
\mathcal{A}\left( z\right) $, we can also write $a_{n}=\mathcal{Z}%
^{-1}\left( \mathcal{A}\left( z\right) \right) $, and we say that the
complex sequence $a_{n}$ is the Inverse $Z$-Transform of the complex
function $\mathcal{A}\left( z\right) $. For example, the sequence $\lambda
^{n}$ (where $\lambda $ is a given non-zero complex number), has $Z$%
-transform%
\begin{equation}
\mathcal{Z}\left( \lambda ^{n}\right) =\sum_{n=0}^{\infty }\frac{\lambda ^{n}%
}{z^{n}}=\frac{1}{1-\frac{\lambda }{z}}=\frac{z}{z-\lambda },  \label{1.7}
\end{equation}%
defined for $\left\vert z\right\vert >\left\vert \lambda \right\vert $. In
particular, the $Z$-transform of the constant sequence $1$ is%
\begin{equation}
\mathcal{Z}\left( 1\right) =\frac{z}{z-1}.  \label{1.8}
\end{equation}

Besides the natural properties of linearity and injectivity, the $Z$%
-transform has the following two important properties to be used in this
work (which formal proofs are easy exercises left to the reader):

\begin{enumerate}
\item (Advance-shifting property) If $\mathcal{Z}\left( a_{n}\right) =%
\mathcal{A}\left( z\right) $, and $k$ is a non-negative integer given, we
have 
\begin{equation}
\mathcal{Z}\left( a_{n+k}\right) =z^{k}\left( \mathcal{A}\left( z\right)
-\sum_{j=0}^{k-1}\frac{a_{j}}{z^{j}}\right) .  \label{1.9}
\end{equation}

\item (Multiplication by the sequence $n$) If $\mathcal{Z}\left(
a_{n}\right) =\mathcal{A}\left( z\right) $, then 
\begin{equation}
\mathcal{Z}\left( na_{n}\right) =-z\frac{d}{dz}\mathcal{A}\left( z\right) .
\label{1.10}
\end{equation}
\end{enumerate}

From (\ref{1.8}) and (\ref{1.10}), we see that the $Z$-transform of the
sequence $n$ is 
\begin{equation}
\mathcal{Z}\left( n\right) =-z\frac{d}{dz}\frac{z}{z-1}=\frac{z}{\left(
z-1\right) ^{2}}.  \label{1.15}
\end{equation}

The $Z$-transform of the sequence $\binom{n}{r}$, where $r$ is a
non-negative integer given, is%
\begin{equation}
\mathcal{Z}\left( \dbinom{n}{r}\right) =\frac{z}{\left( z-1\right) ^{r+1}}.
\label{1.18}
\end{equation}

(The cases $r=0$ and $r=1$ correspond to (\ref{1.8}) and (\ref{1.15}),
respectively. The rest is an easy induction on $r$ left to the reader.)
According to the advance-shifting property (\ref{1.9}), together with (\ref%
{1.18}), we see that for $0\leq k\leq r$, we $Z$-transform of the sequence $%
\binom{n+k}{r}$ is%
\begin{equation}
\mathcal{Z}\left( \dbinom{n+k}{r}\right) =\frac{z^{k+1}}{\left( z-1\right)
^{r+1}}.  \label{1.19}
\end{equation}

In Section \ref{Sec2} we introduce the generalized Stirling numbers of the
second kind $S_{a,b,r}^{\alpha _{s},\beta _{s},r_{s},p_{s}}\left( p,k\right) 
$. The definition we give for $S_{a,b,r}^{\alpha _{s},\beta
_{s},r_{s},p_{s}}\left( p,k\right) $ is related to generalized Eulerian
numbers $A_{a,b,r}^{\alpha _{s},\beta _{s},r_{s},p_{s}}\left( p,k\right) $
(formula (\ref{2.21})), but soon we show that our generalized Stirling
numbers have an explicit formula that generalizes the known explicit formula
for $S\left( p,k\right) $ (proposition \ref{Prop2.1}). In Section \ref{Sec3}
we consider the case $r=r_{s}=1$, and we prove several results involving the
GSN $S_{a,b,1}^{\alpha _{s},\beta _{s},1,p_{s}}\left( p,k\right) $. Some of
these results are proved by induction: we mention that the way we arrived to
them came from a previous work (not included here) with generalized Eulerian
numbers and polynomials \cite{Pi1,Pi2}, together with the connection
described in Section \ref{Sec2}. Finally, in Section \ref{Sec5} we state the
results about the recurrence and the differential operator associated to the
generalized Stirling numbers $S_{a,b,r}^{\alpha _{s},\beta
_{s},r_{s},p_{s}}\left( p,k\right) $ (these results are the main topics of
the second part of this work \cite{Pi3}.) Also, in a future work we consider
the generalized Bell numbers resulting of the generalization of Stirling
numbers of this work \cite{Pi4}.

\section{\label{Sec2}The Generalized Stirling Numbers}

In a recent work \cite{Pi1}, we considered the \textit{Generalized Eulerian
Numbers} (GEN, for short) $A_{a,b,r}^{\alpha _{s},\beta
_{s},r_{s},p_{s}}\left( p,i\right) $, $i=0,1,\ldots ,rp+\sigma $ (where $%
a,b,\alpha _{s},\beta _{s}$ are complex numbers, and $r,p,r_{s},p_{s}$ are
non-negative integers), defined as the coefficients in the expansion of the $%
\left( rp+\sigma \right) $-th degree polynomial $\binom{an+b}{r}%
^{p}\prod_{s=2}^{L}\binom{\alpha _{s}n+\beta _{s}}{r_{s}}^{p_{s}}$ in terms
of the basis $\mathcal{B}_{1}=\left\{ \binom{n+rp+\sigma -i}{rp+\sigma }%
,i=0,1,\ldots ,rp+\sigma \right\} $ (of the vector space $\mathcal{P}%
_{rp+\sigma }$ of $n$-polynomials of degree $\leq rp+\sigma $; see \cite{St}%
, p. 208, 4.3(b)), that is%
\begin{equation}
\binom{an+b}{r}^{p}\prod_{s=2}^{L}\binom{\alpha _{s}n+\beta _{s}}{r_{s}}%
^{p_{s}}=\sum_{i=0}^{rp+\sigma }A_{a,b,r}^{\alpha _{s},\beta
_{s},r_{s},p_{s}}\left( p,i\right) \binom{n+rp+\sigma -i}{rp+\sigma }.
\label{2.1}
\end{equation}

When $r_{s}p_{s}=0$, $s=2,3,\ldots ,L$, we write the GEN as $A_{a,b,r}\left(
p,i\right) $, $i=0,1,\ldots ,rp$. Clearly, the numbers $A_{1,0,1}\left(
p,i\right) $, $i=0,1,\ldots ,p$, correspond to the standard Eulerian
numbers, that we denote as $A\left( p,i\right) $. In this case, expression (%
\ref{2.1}) is just Worpitzky identity \cite{Wo}: $n^{p}=\sum_{i=0}^{p}A%
\left( p,i\right) \binom{n+p-i}{p}$. We have the explicit formula (see \cite%
{Pi1}) 
\begin{equation}
A_{a,b,r}^{\alpha _{s},\beta _{s},r_{s},p_{s}}\left( p,i\right)
=\sum_{j=0}^{i}\left( -1\right) ^{j}\binom{rp+\sigma +1}{j}\binom{a\left(
i-j\right) +b}{r}^{p}\prod_{s=2}^{L}\binom{\alpha _{s}\left( i-j\right)
+\beta _{s}}{r_{s}}^{p_{s}}.  \label{2.4}
\end{equation}

According to (\ref{1.19}), the $Z$-transform of the sequence $\binom{an+b}{r}%
^{p}\prod_{s=2}^{L}\binom{\alpha _{s}n+\beta _{s}}{r_{s}}^{p_{s}}$ is%
\begin{equation}
\mathcal{Z}\left( \binom{an+b}{r}^{p}\prod_{s=2}^{L}\binom{\alpha
_{s}n+\beta _{s}}{r_{s}}^{p_{s}}\right) =\frac{z\sum_{i=0}^{rp+\sigma
}A_{a,b,r}^{\alpha _{s},\beta _{s},r_{s},p_{s}}\left( p,i\right)
z^{rp+\sigma -i}}{\left( z-1\right) ^{rp+\sigma +1}}.  \label{2.101}
\end{equation}

The polynomial%
\begin{equation}
\mathbb{P}_{a,b,r,p}^{\alpha _{s},\beta _{s},r_{s},p_{s}}\left( z\right)
=\sum_{i=0}^{rp+\sigma }A_{a,b,r}^{\alpha _{s},\beta _{s},r_{s},p_{s}}\left(
p,i\right) z^{rp+\sigma -i},  \label{2.2}
\end{equation}%
is the \textit{Generalized Eulerian Polynomial} (GEP, for short). When $%
r_{s}p_{s}=0$, $s=2,3,\ldots ,L$, we write the GEP (\ref{2.2}) as $\mathbb{P}%
_{a,b,r,p}\left( z\right) $. These polynomials were studied in \cite{Pi2}.

Inspired by the well-known case, in which Stirling numbers of the second
kind $S\left( p,k\right) $, $k=0,1,\ldots ,p$, appear in the coefficients of
the expansion of $n^{p}$ in terms of the basis $\mathcal{B}_{2}=\left\{ 
\binom{n}{k},k=0,1,\ldots ,p\right\} $ (of the vector space $\mathcal{P}_{p}$
---of polynomials of degree $\leq p$---), namely $n^{p}=\sum_{k=0}^{p}k!S%
\left( p,k\right) \binom{n}{k}$, we consider the basis $\mathcal{B}%
_{2}=\left\{ \binom{n}{k},k=0,1,\ldots ,rp+\sigma \right\} $ of the vector
space $\mathcal{P}_{rp+\sigma }$ (see \cite{St}, p. 209, 4.3(d)), then write
the $\left( rp+\sigma \right) $-th degree $n$-polynomial $\binom{an+b}{r}%
^{p}\prod_{s=2}^{L}\binom{\alpha _{s}n+\beta _{s}}{r_{s}}^{p_{s}}$ in terms
of $\mathcal{B}_{2}$, and define the\textit{\ Generalized Stirling Numbers
of the Second Kind }(GSN, for short) $S_{a,b,r}^{\alpha _{s},\beta
_{s},r_{s},p_{s}}\left( p,k\right) $, by means of the expansion

\begin{equation}
\binom{an+b}{r}^{p}\prod_{s=2}^{L}\binom{\alpha _{s}n+\beta _{s}}{r_{s}}%
^{p_{s}}=\frac{1}{\left( r!\right) ^{p}\prod_{s=2}^{L}\left( r_{s}!\right)
^{p_{s}}}\sum_{k=0}^{rp+\sigma }k!S_{a,b,r}^{\alpha _{s},\beta
_{s},r_{s},p_{s}}\left( p,k\right) \binom{n}{k}.  \label{2.16}
\end{equation}

If $k<0$ or $k>rp+\sigma $, we have $S_{a,b,r}^{\alpha _{s},\beta
_{s},r_{s},p_{s}}\left( p,k\right) =0$. For the rest of the work,
\textquotedblleft Stirling number(s)\textquotedblright\ will mean
\textquotedblleft Stirling number(s) of the second kind\textquotedblright .

By using (\ref{1.18}) and according to (\ref{2.16}), we see that the $Z$%
-Transform of the sequence $\binom{an+b}{r}^{p}\prod_{s=2}^{L}\binom{\alpha
_{s}n+\beta _{s}}{r_{s}}^{p_{s}}$ is%
\begin{equation}
\mathcal{Z}\left( \binom{an+b}{r}^{p}\prod_{s=2}^{L}\binom{\alpha
_{s}n+\beta _{s}}{r_{s}}^{p_{s}}\right) =\frac{1}{\left( r!\right)
^{p}\prod_{s=2}^{L}\left( r_{s}!\right) ^{p_{s}}}\sum_{k=0}^{rp+\sigma
}k!S_{a,b,r}^{\alpha _{s},\beta _{s},r_{s},p_{s}}\left( p,k\right) \frac{z}{%
\left( z-1\right) ^{k+1}},  \label{2.17}
\end{equation}%
that we can write as%
\begin{equation}
\mathcal{Z}\left( \binom{an+b}{r}^{p}\prod_{s=2}^{L}\binom{\alpha
_{s}n+\beta _{s}}{r_{s}}^{p_{s}}\right) =\frac{z}{\left( z-1\right)
^{rp+\sigma +1}\left( r!\right) ^{p}\prod_{s=2}^{L}\left( r_{s}!\right)
^{p_{s}}}\sum_{k=0}^{rp+\sigma }k!S_{a,b,r}^{\alpha _{s},\beta
_{s},r_{s},p_{s}}\left( p,k\right) \left( z-1\right) ^{rp+\sigma -k}.
\label{2.18}
\end{equation}

Thus, comparing (\ref{2.18}) with (\ref{2.101}), we obtain the following
expression for the GEP $\mathbb{P}_{a,b,r,p}^{\alpha _{s},\beta
_{s},r_{s},p_{s}}\left( z\right) $ in terms of the GSN $S_{a,b,r}^{\alpha
_{s},\beta _{s},r_{s},p_{s}}\left( p,k\right) $,%
\begin{equation}
\mathbb{P}_{a,b,r,p}^{\alpha _{s},\beta _{s},r_{s},p_{s}}\left( z\right) =%
\frac{1}{\left( r!\right) ^{p}\prod_{s=2}^{L}\left( r_{s}!\right) ^{p_{s}}}%
\sum_{k=0}^{rp+\sigma }k!S_{a,b,r}^{\alpha _{s},\beta
_{s},r_{s},p_{s}}\left( p,k\right) \left( z-1\right) ^{rp+\sigma -k}.
\label{2.19}
\end{equation}

The GEP $\mathbb{P}_{a,b,r,p}^{\alpha _{s},\beta _{s},r_{s},p_{s}}\left(
z\right) $ (\ref{2.2}) written in powers of $z-1$, is%
\begin{equation}
\mathbb{P}_{a,b,r,p}^{\alpha _{s},\beta _{s},r_{s},p_{s}}\left( z\right)
=\sum_{k=0}^{rp+\sigma }\sum_{i=0}^{rp+\sigma }\binom{rp+\sigma -i}{%
rp+\sigma -k}A_{a,b,r}^{\alpha _{s},\beta _{s},r_{s},p_{s}}\left( p,i\right)
\left( z-1\right) ^{rp+\sigma -k}.  \label{2.20}
\end{equation}

Thus, from (\ref{2.19}) and (\ref{2.20}), we see that for $0\leq k\leq
rp+\sigma $ we have%
\begin{equation}
S_{a,b,r}^{\alpha _{s},\beta _{s},r_{s},p_{s}}\left( p,k\right) =\frac{%
\left( r!\right) ^{p}\prod_{s=2}^{L}\left( r_{s}!\right) ^{p_{s}}}{k!}%
\sum_{i=0}^{rp+\sigma }\binom{rp+\sigma -i}{rp+\sigma -k}A_{a,b,r}^{\alpha
_{s},\beta _{s},r_{s},p_{s}}\left( p,i\right) .  \label{2.21}
\end{equation}

Reciprocally, we can write (\ref{2.18}) as%
\begin{eqnarray}
&&\mathcal{Z}\left( \binom{an+b}{r}^{p}\prod_{s=2}^{L}\binom{\alpha
_{s}n+\beta _{s}}{r_{s}}^{p_{s}}\right)  \label{2.22} \\
&=&\frac{z}{\left( z-1\right) ^{rp+\sigma +1}}\sum_{k=0}^{rp+\sigma }\frac{k!%
}{\left( r!\right) ^{p}\prod_{s=2}^{L}\left( r_{s}!\right) ^{p_{s}}}%
S_{a,b,r}^{\alpha _{s},\beta _{s},r_{s},p_{s}}\left( p,k\right)
\sum_{j=0}^{rp+\sigma -k}\binom{rp+\sigma -k}{j}\left( -1\right)
^{j}z^{rp+\sigma -k-j}.  \notag
\end{eqnarray}

In the right-hand side of (\ref{2.22}) introduce the new summation index $%
i=k+j$, to obtain 
\begin{eqnarray}
&&\mathcal{Z}\left( \binom{an+b}{r}^{p}\prod_{s=2}^{L}\binom{\alpha
_{s}n+\beta _{s}}{r_{s}}^{p_{s}}\right)  \label{2.23} \\
&=&\frac{z}{\left( z-1\right) ^{rp+\sigma +1}}\sum_{i=0}^{rp+\sigma }\left( 
\frac{\left( -1\right) ^{i}}{\left( r!\right) ^{p}\prod_{s=2}^{L}\left(
r_{s}!\right) ^{p_{s}}}\sum_{k=0}^{rp+\sigma }\binom{rp+\sigma -k}{rp+\sigma
-i}\left( -1\right) ^{k}k!S_{a,b,r}^{\alpha _{s},\beta
_{s},r_{s},p_{s}}\left( p,k\right) \right) z^{rp+\sigma -i}.  \notag
\end{eqnarray}

Thus, comparing (\ref{2.23}) with (\ref{2.101}) (and (\ref{2.2})), we see
that for $0\leq i\leq rp+\sigma $ we have%
\begin{equation}
A_{a,b,r}^{\alpha _{s},\beta _{s},r_{s},p_{s}}\left( p,i\right) =\frac{%
\left( -1\right) ^{i}}{\left( r!\right) ^{p}\prod_{s=2}^{L}\left(
r_{s}!\right) ^{p_{s}}}\sum_{k=0}^{rp+\sigma }\binom{rp+\sigma -k}{rp+\sigma
-i}\left( -1\right) ^{k}k!S_{a,b,r}^{\alpha _{s},\beta
_{s},r_{s},p_{s}}\left( p,k\right) .  \label{2.24}
\end{equation}

In the standard case $a=1,b=0,r=1$, $r_{s}p_{s}=0$, formulas (\ref{2.21})
and (\ref{2.24}) are well-known results: $S\left( p,k\right) =\frac{1}{k!}%
\sum_{i=0}^{p}\binom{p-i}{p-k}A\left( p,i\right) ,$ and $A\left( p,i\right)
=\left( -1\right) ^{i}\sum_{k=0}^{p}\binom{p-k}{p-i}\left( -1\right)
^{k}k!S\left( p,k\right) .$(See \cite{G-K-P}, p. 269, formulas (6.39) and
(6.40).)

The GSN $S_{a,b,r}^{\alpha _{s},\beta _{s},r_{s},p_{s}}\left( p,k\right) $
have an explicit formula (generalizing the known explicit formula $S\left(
p,k\right) =\frac{1}{k!}\sum_{j=0}^{k}\left( -1\right) ^{j}\binom{k}{j}%
\left( k-j\right) ^{p}$).

\begin{proposition}
\label{Prop2.1}We have the following explicit formula for the GSN $%
S_{a,b,r}^{\alpha _{s},\beta _{s},r_{s},p_{s}}\left( p,k\right) $%
\begin{equation}
S_{a,b,r}^{\alpha _{s},\beta _{s},r_{s},p_{s}}\left( p,k\right) =\frac{%
\left( r!\right) ^{p}\prod_{s=2}^{l}\left( r_{s}!\right) ^{p_{s}}}{k!}%
\sum_{j=0}^{k}\left( -1\right) ^{j}\binom{k}{j}\binom{\left( k-j\right) a+b}{%
r}^{p}\prod_{s=2}^{L}\binom{\left( k-j\right) \alpha _{s}+\beta _{s}}{r_{s}}%
^{p_{s}}.  \label{2.25}
\end{equation}
\end{proposition}

\begin{proof}
From (\ref{2.21}) and (\ref{2.25}), we have to show that%
\begin{equation}
\sum_{i=0}^{k}\binom{rp+\sigma -i}{k-i}A_{a,b,r}^{\alpha _{s},\beta
_{s},r_{s},p_{s}}\left( p,i\right) =\sum_{j=0}^{k}\left( -1\right) ^{j}%
\binom{k}{j}\binom{\left( k-j\right) a+b}{r}^{p}\prod_{s=2}^{L}\binom{\left(
k-j\right) \alpha _{s}+\beta _{s}}{r_{s}}^{p_{s}}.  \label{2.26}
\end{equation}

Beginning with the left-hand side of (\ref{2.26}), we use the explicit
formula (\ref{2.4}) for $A_{a,b,r}^{\left( \alpha _{s},\beta
_{s},r_{s},p_{s}\right) }$, to write%
\begin{eqnarray}
&&\sum_{i=0}^{k}\binom{rp+\sigma -i}{k-i}A_{a,b,r}^{\left( \alpha _{s},\beta
_{s},r_{s},p_{s}\right) }\left( p,i\right)  \notag \\
&=&\sum_{i=0}^{k}\binom{rp+\sigma -i}{k-i}\sum_{j=0}^{k}\left( -1\right)
^{i+j}\binom{rp+\sigma +1}{i-j}\binom{aj+b}{r}^{p}\prod_{s=2}^{L}\binom{%
\alpha _{s}j+\beta _{s}}{r_{s}}^{p_{s}}  \notag \\
&=&\sum_{j=0}^{k}\left( -1\right) ^{j}\sum_{i=0}^{k}\left( -1\right) ^{k+i}%
\binom{rp+\sigma -i}{k-i}\binom{rp+\sigma +1}{i-k+j}\binom{a\left(
k-j\right) +b}{r}^{p}\prod_{s=2}^{L}\binom{\alpha _{s}\left( k-j\right)
+\beta _{s}}{r_{s}}^{p_{s}}.  \label{2.27}
\end{eqnarray}

Now, observe that for $0\leq j\leq k$, we have%
\begin{eqnarray*}
\sum_{i=0}^{k}\left( -1\right) ^{k+i}\binom{rp+\sigma -i}{k-i}\binom{%
rp+\sigma +1}{i-k+j} &=&\sum_{i=0}^{k}\binom{-rp-\sigma +i+k-i-1}{k-i}\binom{%
rp+\sigma +1}{i-k+j} \\
&=&\sum_{i=0}^{j}\binom{k-rp-\sigma -1}{i}\binom{rp+\sigma +1}{j-i} \\
&=&\binom{k}{j},
\end{eqnarray*}%
where the last step is the Vandermonde Convolution. Thus, expression (\ref%
{2.27}) (the left-hand side of (\ref{2.26})) is precisely the right-hand
side of (\ref{2.26}), as desired.
\end{proof}

From (\ref{2.25}) one can easily see that, for $p$ given, the first two
values $S_{a,b,r}^{\alpha _{s},\beta _{s},r_{s},p_{s}}\left( p,0\right) $
and $S_{a,b,r}^{\alpha _{s},\beta _{s},r_{s},p_{s}}\left( p,1\right) $ are 
\begin{equation}
S_{a,b,r}^{\alpha _{s},\beta _{s},r_{s},p_{s}}\left( p,0\right) =\left( r!%
\binom{b}{r}\right) ^{p}\prod\limits_{s=2}^{L}\left( r_{s}!\binom{\beta _{s}%
}{r_{s}}\right) ^{p_{s}},  \label{2.29}
\end{equation}

\begin{equation}
S_{a,b,r}^{\alpha _{s},\beta _{s},r_{s},p_{s}}\left( p,1\right) =\left( r!%
\binom{a+b}{r}\right) ^{p}\prod_{s=2}^{L}\left( r_{s}!\binom{\alpha
_{s}+\beta _{s}}{r_{s}}\right) ^{p_{s}}-S_{a,b,r}^{\alpha _{s},\beta
_{s},r_{s},p_{s}}\left( p,0\right) ,  \label{2.291}
\end{equation}%
and the last value $S_{a,b,r}^{\alpha _{s},\beta _{s},r_{s},p_{s}}\left(
p,rp+\sigma \right) $ is 
\begin{equation}
S_{a,b,r}^{\alpha _{s},\beta _{s},r_{s},p_{s}}\left( p,rp+\sigma \right)
=a^{rp}\prod_{s=2}^{L}\alpha _{s}{}^{r_{s}p_{s}}.  \label{2.30}
\end{equation}

\section{\label{Sec3}Case $r=r_{s}=1$: Main Results}

In this section we consider the GSN $S_{a,b,1}^{\alpha _{s},\beta
_{s},1,p_{s}}\left( p,k\right) $, $k=0,1,\ldots ,p+\sigma $, $\sigma
=p_{2}+\cdots +p_{L}$, involved in the expansion $\left( an+b\right)
^{p}\prod\limits_{s=2}^{L}\left( \alpha _{s}n+\beta _{s}\right)
^{p_{s}}=\sum_{k=0}^{p+\sigma }k!S_{a,b,1}^{\alpha _{s},\beta
_{s},1,p_{s}}\left( p,k\right) \!\binom{n}{k}$. For the sake of a simpler
notation, we will work with the GSN $S_{a_{1},b_{1},1}^{a_{2},b_{2},1,p_{2}}%
\left( p,k\right) $, that we write as $S_{a_{1},b_{1}}^{a_{2},b_{2},p_{2}}%
\left( p,k\right) $. According to (\ref{2.25}) we have the explicit formula 
\begin{equation}
S_{a_{1},b_{1}}^{a_{2},b_{2},p_{2}}\left( p_{1},k\right) =\frac{1}{k!}%
\sum_{j=0}^{k}\left( -1\right) ^{j}\binom{k}{j}\left( a_{1}\left( k-j\right)
+b_{1}\right) ^{p_{1}}\left( a_{2}\left( k-j\right) +b_{2}\right) ^{p_{2}}.
\label{3.1}
\end{equation}

(The case $p_{2}=0$ was studied before by L. Verde-Star \cite{V-S}, and it
is particular case of the Hsu and Shiue generalization \cite{H-S1}.) We
present some Generalized Stirling Number Triangles (GSNT, for short), where
the GSN $S_{a_{1},b_{1}}^{a_{2},b_{2},p_{2}}\left( p,k\right) $ appear in a
triangular array ($p\in \mathbb{N}$ stands for lines and $k$ stands for
columns, $0\leq k\leq p+p_{2}$).

\begin{equation*}
\begin{array}{lll}
\begin{tabular}{|c|c|c|c|c|c|}
\hline
\multicolumn{6}{|c|}{\textbf{GSNT1: }$S_{1,1}^{1,0,1}\left( p,k\right) .$}
\\ \hline
$p\diagdown k$ & $0$ & $1$ & $2$ & $3$ & $\cdots $ \\ \hline
$0$ & $0$ & $1$ &  &  &  \\ \hline
$1$ & $0$ & $2$ & $1$ &  & $\cdots $ \\ \hline
$2$ & $0$ & $4$ & $5$ & $1$ &  \\ \hline
$\vdots $ &  & $\vdots $ &  & $\vdots $ &  \\ \hline
\end{tabular}
&  & 
\begin{tabular}{|c|c|c|c|c|cc|}
\hline
\multicolumn{6}{|c}{\textbf{GSNT2: }$S_{1,1}^{1,0,2}\left( p,k\right) .$} & 
\\ \hline
$p\diagdown k$ & $0$ & $1$ & $2$ & $3$ & $4$ & \multicolumn{1}{|c|}{$\cdots $%
} \\ \hline
$0$ & $0$ & $1$ & $1$ &  &  & \multicolumn{1}{|c|}{} \\ \hline
$1$ & $0$ & $2$ & $4$ & $1$ &  & \multicolumn{1}{|c|}{$\cdots $} \\ \hline
$2$ & $0$ & $4$ & $14$ & $8$ & $1$ & \multicolumn{1}{|c|}{} \\ \hline
$\vdots $ &  & $\vdots $ &  & $\vdots $ &  & \multicolumn{1}{|c|}{} \\ \hline
\end{tabular}%
\end{array}%
\end{equation*}

\begin{equation*}
\begin{array}{lll}
\begin{tabular}{|c|c|c|c|cc|}
\hline
\multicolumn{5}{|c}{\textbf{GSNT3: }$S_{1,2}^{1,1,1}\left( p,k\right) .$} & 
\\ \hline
$p\diagdown k$ & $0$ & $1$ & $2$ & $3$ & \multicolumn{1}{|c|}{$\cdots $} \\ 
\hline
$0$ & $1$ & $1$ &  &  & \multicolumn{1}{|c|}{} \\ \hline
$1$ & $2$ & $4$ & $1$ &  & \multicolumn{1}{|c|}{$\cdots $} \\ \hline
$2$ & $4$ & $14$ & $8$ & $1$ & \multicolumn{1}{|c|}{} \\ \hline
$\vdots $ &  & $\vdots $ &  & $\vdots $ & \multicolumn{1}{|c|}{} \\ \hline
\end{tabular}
&  & 
\begin{tabular}{|c|c|c|c|c|cc|}
\hline
\multicolumn{6}{|c}{\textbf{GSNT4: }$S_{1,2}^{1,1,2}\left( p,k\right) .$} & 
\\ \hline
$p\diagdown k$ & $0$ & $1$ & $2$ & $3$ & $4$ & \multicolumn{1}{|c|}{$\cdots $%
} \\ \hline
$0$ & $1$ & $3$ & $1$ &  &  & \multicolumn{1}{|c|}{} \\ \hline
$1$ & $2$ & $10$ & $7$ & $1$ &  & \multicolumn{1}{|c|}{$\cdots $} \\ \hline
$2$ & $4$ & $32$ & $38$ & $12$ & $1$ & \multicolumn{1}{|c|}{} \\ \hline
$\vdots $ &  & $\vdots $ &  & $\vdots $ &  & \multicolumn{1}{|c|}{} \\ \hline
\end{tabular}%
\end{array}%
\end{equation*}

The following facts are obvious from (\ref{3.1}):%
\begin{equation}
\begin{array}{lll}
S_{a_{1},b_{1}}^{a_{1},b_{1},p_{2}}\left( p_{1},k\right)
=S_{a_{1},b_{1}}\left( p_{1}+p_{2},k\right) , &  & 
S_{a_{1},b_{1}}^{a_{2},b_{2},p_{2}}\left( 0,0\right) =b_{2}^{p_{2}}, \\ 
S_{a_{1},b_{1}}^{a_{2},b_{2},p_{2}}\left( p_{1},k\right)
=S_{a_{2},b_{2}}^{a_{1},b_{1},p_{1}}\left( p_{2},k\right) , &  & 
S_{a_{1},b_{1}}^{a_{2},b_{2},p_{2}}\left( 0,k\right) =S_{a_{2},b_{2}}\left(
p_{2},k\right) .%
\end{array}
\label{3.2}
\end{equation}

Also, one can see easily from (\ref{3.1}) that 
\begin{eqnarray}
S_{1,1}\left( p,k\right) &=&S\left( p+1,k+1\right) ,  \label{3.3} \\
S_{1,2}\left( p,k\right) &=&S\left( p+2,k+2\right) -S\left( p+1,k+2\right) .
\notag
\end{eqnarray}

(We will be using (\ref{3.2}) and (\ref{3.3}) without further comments.)

Some values of the GSN $S_{a_{1},b_{1}}^{a_{2},b_{2},p_{2}}\left(
p_{1},k\right) $ are%
\begin{eqnarray}
S_{a_{1},b_{1}}^{a_{2},b_{2},p_{2}}\left( p_{1},0\right)
&=&\!\!b_{1}^{p_{1}}\!b_{2}^{p_{2}}\!,  \label{3.4} \\
S_{a_{1},b_{1}}^{a_{2},b_{2},p_{2}}\left( p_{1},1\right) &=&\left(
a_{1}+b_{1}\right) ^{p_{1}}\!\left( a_{2}+b_{2}\right)
^{p_{2}}\!-b_{1}^{p_{1}}\!b_{2}^{p_{2}}\!,  \notag \\
S_{a_{1},b_{1}}^{a_{2},b_{2},p_{2}}\left( p_{1},2\right) &=&\!\!\frac{1}{2}%
\left( 2a_{1}+b_{1}\right) ^{p_{1}}\!\left( 2a_{2}+b_{2}\right)
^{p_{2}}\!-\left( a_{1}+b_{1}\right) ^{p_{1}}\!\left( a_{2}+b_{2}\right)
^{p_{2}}\!+\frac{1}{2}b_{1}^{p_{1}}\!b_{2}^{p_{2}}\!,  \notag \\
&&\vdots  \notag \\
S_{a_{1},b_{1}}^{a_{2},b_{2},p_{2}}\left( p_{1},p_{1}+p_{2}\right)
&=&a_{1}^{p_{1}}a_{2}^{p_{2}}.  \notag
\end{eqnarray}

The GSN $S_{a_{1},b_{1}}^{a_{2},b_{2},p_{2}}\left( p_{1},k\right) $ satisfy
the recurrence (to be proved in \cite{Pi3}) 
\begin{equation}
S_{a_{1},b_{1}}^{a_{2},b_{2},p_{2}}\left( p_{1},k\right)
=a_{1}S_{a_{1},b_{1}}^{a_{2},b_{2},p_{2}}\left( p_{1}-1,k-1\right) +\left(
a_{1}k+b_{1}\right) S_{a_{1},b_{1}}^{a_{2},b_{2},p_{2}}\left(
p_{1}-1,k\right) .  \label{3.5}
\end{equation}

\begin{lemma}
\label{Lemma1}The GSN $S_{a_{1},b_{1}}^{a_{2},b_{2},p_{2}}\left(
p_{1},k\right) $ are related to the GSN $S_{c_{1},d_{1}}^{c_{2},d_{2},j_{2}}%
\left( p_{1},k\right) $, by the following formula%
\begin{equation}
S_{a_{1},b_{1}}^{a_{2},b_{2},p_{2}}\left( p_{1},k\right)
=c_{1}^{-p_{1}}c_{2}^{-p_{2}}\sum_{j_{1}=0}^{p_{1}}\sum_{j_{2}=0}^{p_{2}}%
\binom{p_{1}}{j_{1}}\binom{p_{2}}{j_{2}}a_{1}^{j_{1}}a_{2}^{j_{2}}\left(
b_{1}c_{1}-a_{1}d_{1}\right) ^{p_{1}-j_{1}}\left(
b_{2}c_{2}-a_{2}d_{2}\right)
^{p_{2}-j_{2}}S_{c_{1},d_{1}}^{c_{2},d_{2},j_{2}}\left( j_{1},k\right) .
\label{3.6}
\end{equation}
\end{lemma}

\begin{proof}
It is a straightforward calculation:%
\begin{eqnarray*}
&&c_{1}^{-p_{1}}c_{2}^{-p_{2}}\sum_{j_{1}=0}^{p_{1}}\sum_{j_{2}=0}^{p_{2}}%
\binom{p_{1}}{j_{1}}\binom{p_{2}}{j_{2}}a_{1}^{j_{1}}a_{2}^{j_{2}}\left(
b_{1}c_{1}-a_{1}d_{1}\right) ^{p_{1}-j_{1}}\left(
b_{2}c_{2}-a_{2}d_{2}\right)
^{p_{2}-j_{2}}S_{c_{1},d_{1}}^{c_{2},d_{2},j_{2}}\left( j_{1},k\right) \\
&=&c_{1}^{-p_{1}}c_{2}^{-p_{2}}\sum_{j_{1}=0}^{p_{1}}\sum_{j_{2}=0}^{p_{2}}%
\binom{p_{1}}{j_{1}}\binom{p_{2}}{j_{2}}a_{1}^{j_{1}}a_{2}^{j_{2}}\left(
b_{1}c_{1}-a_{1}d_{1}\right) ^{p_{1}-j_{1}}\left(
b_{2}c_{2}-a_{2}d_{2}\right) ^{p_{2}-j_{2}}\times \\
&&\times \frac{1}{k!}\sum_{l=0}^{k}\left( -1\right) ^{l}\binom{k}{l}\left(
c_{1}\left( k-l\right) +d_{1}\right) ^{j_{1}}\left( c_{2}\left( k-l\right)
+d_{2}\right) ^{j_{2}} \\
&=&\frac{1}{k!}\sum_{l=0}^{k}\left( -1\right) ^{l}\binom{k}{l}%
c_{1}^{-p_{1}}\sum_{j_{1}=0}^{p_{1}}\binom{p_{1}}{j_{1}}\left(
b_{1}c_{1}-a_{1}d_{1}\right) ^{p_{1}-j_{1}}\left( a_{1}c_{1}\left(
k-l\right) +a_{1}d_{1}\right) ^{j_{1}} \\
&&\times c_{2}^{-p_{2}}\sum_{j_{2}=0}^{p_{2}}\binom{p_{2}}{j_{2}}\left(
b_{2}c_{2}-a_{2}d_{2}\right) ^{p_{2}-j_{2}}\left( a_{2}c_{2}\left(
k-l\right) +a_{2}d_{2}\right) ^{j_{2}} \\
&=&\frac{1}{k!}\sum_{l=0}^{k}\left( -1\right) ^{l}\binom{k}{l}\left(
a_{1}\left( k-l\right) +b_{1}\right) ^{p_{1}}\left( a_{2}\left( k-l\right)
+b_{2}\right) ^{p_{2}} \\
&=&S_{a_{1},b_{1}}^{a_{2},b_{2},p_{2}}\left( p_{1},k\right) ,
\end{eqnarray*}%
as desired.
\end{proof}

In particular, from (\ref{3.6}) (with $c_{1}=c_{2}=1,d_{1}=d_{2}=0$) we see
that the GSN $S_{a_{1},b_{1}}^{a_{1},b_{1},p_{2}}\left( p_{1},k\right) $ can
be written in terms of standard Stirling numbers as%
\begin{equation}
S_{a_{1},b_{1}}^{a_{2},b_{2},p_{2}}\left( p_{1},k\right)
=\sum_{j_{1}=0}^{p_{1}}\sum_{j_{2}=0}^{p_{2}}\binom{p_{1}}{j_{1}}\binom{p_{2}%
}{j_{2}}a_{1}^{j_{1}}a_{2}^{j_{2}}b_{1}^{p_{1}-j_{1}}b_{2}^{p_{2}-j_{2}}S%
\left( j_{1}+j_{2},k\right) .  \label{3.7}
\end{equation}

If $a_{1}=a_{2},b_{1}=b_{2}$ (or if $p_{2}=0$), expression (\ref{3.7})
reduces to%
\begin{equation}
S_{a,b}\left( p,k\right) =\sum_{j=0}^{p}\binom{p}{j}a^{j}b^{p-j}S\left(
j,k\right) .  \label{3.71}
\end{equation}

(The particular case $a=b=1$ of (\ref{3.71}) is the known formula $S\left(
p+1,k+1\right) =\sum_{j=0}^{p}\binom{p}{j}S\left( j,k\right) .$) Similarly,
the standard Stirling numbers can be written in terms of GSN (from (\ref{3.6}%
) with $a_{1}=a_{2}=1,b_{1}=b_{2}=0$) as%
\begin{equation}
S\left( p_{1}+p_{2},k\right)
=c_{1}^{-p_{1}}c_{2}^{-p_{2}}\sum_{j_{1}=0}^{p_{1}}\sum_{j_{2}=0}^{p_{2}}%
\binom{p_{1}}{j_{1}}\binom{p_{2}}{j_{2}}\left( -d_{1}\right)
^{p_{1}-j_{1}}\left( -d_{2}\right)
^{p_{2}-j_{2}}S_{c_{1},d_{1}}^{c_{2},d_{2},j_{2}}\left( j_{1},k\right) .
\label{3.8}
\end{equation}

If $c_{1}=c_{2},d_{1}=d_{2}$ (or if $p_{2}=0$), expression (\ref{3.8})
reduces to%
\begin{equation}
S\left( p,k\right) =c^{-p}\sum_{j=0}^{p}\binom{p}{j}\left( -d\right)
^{p-j}S_{c,d}\left( j,k\right) .  \label{3.81}
\end{equation}

(The particular case $c=d=1$ of (\ref{3.81}) is the known formula $S\left(
p,k\right) =\sum_{j=0}^{p}\binom{p}{j}\left( -1\right) ^{p-j}S\left(
j+1,k+1\right) .$)

\begin{lemma}
\label{Lemma2}We have%
\begin{equation}
S_{a_{1},a_{1}+b_{1}}^{a_{2},a_{2}+b_{2},p_{2}}\left( p_{1},k\right)
=S_{a_{1},b_{1}}^{a_{2},b_{2},p_{2}}\left( p_{1},k\right) +\left( k+1\right)
S_{a_{1},b_{1}}^{a_{2},b_{2},p_{2}}\left( p_{1},k+1\right) .  \label{3.9}
\end{equation}
\end{lemma}

\begin{proof}
By using (\ref{3.6}) we can write the GSN $%
S_{a_{1},a_{1}+b_{1}}^{a_{2},a_{2}+b_{2},p_{2}}\left( p_{1},k\right) $
(left-hand side of (\ref{3.9})) as%
\begin{equation*}
S_{a_{1},a_{1}+b_{1}}^{a_{2},a_{2}+b_{2},p_{2}}\left( p_{1},k\right)
=\sum_{j_{1}=0}^{p_{1}}\sum_{j_{2}=0}^{p_{2}}\binom{p_{1}}{j_{1}}\binom{p_{2}%
}{j_{2}}%
a_{1}^{p_{1}-j_{1}}a_{2}^{p_{2}-j_{2}}S_{a_{1},b_{1}}^{a_{2},b_{2},j_{2}}%
\left( j_{1},k\right) ,
\end{equation*}%
that is%
\begin{eqnarray}
&&S_{a_{1},a_{1}+b_{1}}^{a_{2},a_{2}+b_{2},p_{2}}\left( p_{1},k\right) 
\notag \\
&=&\sum_{j_{1}=0}^{p_{1}}\sum_{j_{2}=0}^{p_{2}}\binom{p_{1}}{j_{1}}\binom{%
p_{2}}{j_{2}}a_{1}^{p_{1}-j_{1}}a_{2}^{p_{2}-j_{2}}\frac{1}{k!}%
\sum_{t=0}^{k}\left( -1\right) ^{t}\binom{k}{t}\left( a_{1}\left( k-t\right)
+b_{1}\right) ^{j_{1}}\left( a_{2}\left( k-t\right) +b_{2}\right) ^{j_{2}} 
\notag \\
&=&\frac{1}{k!}\sum_{t=0}^{k}\left( -1\right) ^{t}\binom{k}{t}\left(
a_{1}\left( k+1-t\right) +b_{1}\right) ^{p_{1}}\left( a_{2}\left(
k+1-t\right) +b_{2}\right) ^{p_{2}}.  \label{3.10}
\end{eqnarray}

On the other hand, the right-hand side of (\ref{3.9}) is%
\begin{eqnarray*}
&&S_{a_{1},b_{1}}^{a_{2},b_{2},p_{2}}\left( p_{1},k\right) +\left(
k+1\right) S_{a_{1},b_{1}}^{a_{2},b_{2},p_{2}}\left( p_{1},k+1\right) \\
&=&\frac{1}{k!}\sum_{t=0}^{k}\left( -1\right) ^{t}\binom{k}{t}\left(
a_{1}\left( k-t\right) +b_{1}\right) ^{p_{1}}\left( a_{2}\left( k-t\right)
+b_{2}\right) ^{p_{2}} \\
&&+\left( k+1\right) \frac{1}{\left( k+1\right) !}\sum_{t=0}^{k+1}\left(
-1\right) ^{t}\binom{k+1}{t}\left( a_{1}\left( k+1-t\right) +b_{1}\right)
^{p_{1}}\left( a_{2}\left( k+1-t\right) +b_{2}\right) ^{p_{2}} \\
&=&\frac{1}{k!}\sum_{t=0}^{k}\left( -1\right) ^{t}\binom{k}{t}\left(
a_{1}\left( k-t\right) +b_{1}\right) ^{p_{1}}\left( a_{2}\left( k-t\right)
+b_{2}\right) ^{p_{2}} \\
&&+\frac{1}{k!}\sum_{t=0}^{k+1}\left( -1\right) ^{t}\left( \binom{k}{t}+%
\binom{k}{t-1}\right) \left( a_{1}\left( k+1-t\right) +b_{1}\right)
^{p_{1}}\left( a_{2}\left( k+1-t\right) +b_{2}\right) ^{p_{2}} \\
&=&\frac{1}{k!}\sum_{t=0}^{k}\left( \left( -1\right) ^{t}+\left( -1\right)
^{t+1}\right) \binom{k}{t}\left( a_{1}\left( k-t\right) +b_{1}\right)
^{p_{1}}\left( a_{2}\left( k-t\right) +b_{2}\right) ^{p_{2}} \\
&&+\frac{1}{k!}\sum_{t=0}^{k}\left( -1\right) ^{t}\binom{k}{t}\left(
a_{1}\left( k+1-t\right) +b_{1}\right) ^{p_{1}}\left( a_{2}\left(
k+1-t\right) +b_{2}\right) ^{p_{2}} \\
&=&\frac{1}{k!}\sum_{t=0}^{k}\left( -1\right) ^{t}\binom{k}{t}\left(
a_{1}\left( k+1-t\right) +b_{1}\right) ^{p_{1}}\left( a_{2}\left(
k+1-t\right) +b_{2}\right) ^{p_{2}},
\end{eqnarray*}%
which is equal to (\ref{3.10}), as desired.
\end{proof}

(The case $p_{2}=0,a_{1}=1,b_{1}=0$ of (\ref{3.9}) is the known recurrence
for standard Stirling numbers.)

Formula (\ref{3.6}) is just a particular case ($m=0$) of the following more
general result.

\begin{proposition}
\label{Prop1}For $0\leq k\leq p_{1}+p_{2}$, and any non-negative integer $m$
we have%
\begin{eqnarray}
&&k!S_{a_{1},b_{1}}^{a_{2},b_{2},p_{2}}\left( p_{1},k\right)  \label{3.11} \\
&=&c_{1}^{-p_{1}}c_{2}^{-p_{2}}\sum_{j_{1}=0}^{p_{1}}\sum_{j_{2}=0}^{p_{2}}%
\binom{p_{1}}{j_{1}}\binom{p_{2}}{j_{2}}a_{1}^{j_{1}}a_{2}^{j_{2}}\left(
b_{1}c_{1}-a_{1}c_{1}m-a_{1}d_{1}\right) ^{p_{1}-j_{1}}\left(
b_{2}c_{2}-a_{2}c_{2}m-a_{2}d_{2}\right) ^{p_{2}-j_{2}}\times  \notag \\
&&\times \sum_{t=0}^{m}\binom{m}{t}\left( k+t\right)
!S_{c_{1},d_{1}}^{c_{2},d_{2},j_{2}}\left( j_{1},k+t\right) .  \notag
\end{eqnarray}
\end{proposition}

\begin{proof}
If $k=p_{1}+p_{2}$, the left-hand side of (\ref{3.11}) is $\left(
p_{1}+p_{2}\right) !a_{1}^{p_{1}}a_{2}^{p_{2}}$, and the right-hand side of (%
\ref{3.11}) reduces to the only term $%
c_{1}^{-p_{1}}c_{2}^{-p_{2}}a_{1}^{p_{1}}a_{2}^{p_{2}}\left(
p_{1}+p_{2}\right) !c_{1}^{p_{1}}c_{2}^{p_{2}}=\left( p_{1}+p_{2}\right)
!a_{1}^{p_{1}}a_{2}^{p_{2}}$, so we can suppose that $0\leq k<p_{1}+p_{2}$.
Observe that the right-hand side of (\ref{3.11}) can be written as%
\begin{eqnarray*}
&&c_{1}^{-p_{1}}c_{2}^{-p_{2}}\sum_{j_{1}=0}^{p_{1}}\sum_{j_{2}=0}^{p_{2}}%
\binom{p_{1}}{j_{1}}\binom{p_{2}}{j_{2}}a_{1}^{j_{1}}a_{2}^{j_{2}}\left(
b_{1}c_{1}-a_{1}c_{1}m-a_{1}d_{1}\right) ^{p_{1}-j_{1}}\left(
b_{2}c_{2}-a_{2}c_{2}m-a_{2}d_{2}\right) ^{p_{2}-j_{2}}\times \\
&&\times \sum_{t=0}^{m}\binom{m}{t}\sum_{l=0}^{k+t}\left( -1\right) ^{l}%
\binom{k+t}{l}\left( c_{1}\left( k+t-l\right) +d_{1}\right) ^{j_{1}}\left(
c_{2}\left( k+t-l\right) +d_{2}\right) ^{j_{2}} \\
&=&\sum_{t=0}^{m}\binom{m}{t}\sum_{l=0}^{k+t}\left( -1\right) ^{l}\binom{k+t%
}{l}\left( a_{1}\left( k+t-l\right) +b_{1}-a_{1}m\right) ^{p_{1}}\left(
a_{2}\left( k+t-l\right) +b_{2}-a_{2}m\right) ^{p_{2}}.
\end{eqnarray*}

That is, we have to show that for $0\leq k<p_{1}+p_{2}$, and any
non-negative integer $m$, one has%
\begin{eqnarray}
&&\sum_{t=0}^{m}\binom{m}{t}\sum_{l=0}^{k+t}\left( -1\right) ^{l}\binom{k+t}{%
l}\left( a_{1}\left( k+t-l\right) +b_{1}-a_{1}m\right) ^{p_{1}}\left(
a_{2}\left( k+t-l\right) +b_{2}-a_{2}m\right) ^{p_{2}}  \notag \\
&=&k!S_{a_{1},b_{1}}^{a_{2},b_{2},p_{2}}\left( p_{1},k\right) .  \label{3.12}
\end{eqnarray}

We proceed by induction on $m$. For $m=0$ formula (\ref{3.12}) is trivial.
If we suppose (\ref{3.12}) is true for a given $m\in \mathbb{N}$, then%
\begin{eqnarray*}
&&\sum_{t=0}^{m+1}\binom{m+1}{t}\sum_{l=0}^{k+t}\left( -1\right) ^{l}\binom{%
k+t}{l}\times \\
&&\times \left( a_{1}\left( k+t-l\right) +b_{1}-a_{1}\left( m+1\right)
\right) ^{p_{1}}\left( a_{2}\left( k+t-l\right) +b_{2}-a_{2}\left(
m+1\right) \right) ^{p_{2}} \\
&=&\sum_{t=0}^{m}\binom{m}{t}\sum_{l=0}^{k+t}\left( -1\right) ^{l}\binom{k+t%
}{l}\times \\
&&\times \left( a_{1}\left( k+t-l\right) +b_{1}-a_{1}\left( m+1\right)
\right) ^{p_{1}}\left( a_{2}\left( k+t-l\right) +b_{2}-a_{2}\left(
m+1\right) \right) ^{p_{2}} \\
&&+\sum_{t=0}^{m}\binom{m}{t}\sum_{l=0}^{k+t+1}\left( -1\right) ^{l}\binom{%
k+t+1}{l}\times \\
&&\times \left( a_{1}\left( k+t+1-l\right) +b_{1}-a_{1}\left( m+1\right)
\right) ^{p_{1}}\left( a_{2}\left( k+t+1-l\right) +b_{2}-a_{2}\left(
m+1\right) \right) ^{p_{2}} \\
&=&k!S_{a_{1},b_{1}-a_{1}}^{a_{2},b_{2}-a_{2},p_{2}}\left( p_{1},k\right)
+\left( k+1\right) !S_{a_{1},b_{1}-a_{1}}^{a_{2},b_{2}-a_{2},p_{2}}\left(
p_{1},k+1\right) \\
&=&k!\left( S_{a_{1},b_{1}-a_{1}}^{a_{2},b_{2}-a_{2},p_{2}}\left(
p_{1},k\right) +\left( k+1\right)
S_{a_{1},b_{1}-a_{1}}^{a_{2},b_{2}-a_{2},p_{2}}\left( p_{1},k+1\right)
\right) \\
&=&k!S_{a_{1},b_{1}}^{a_{2},b_{2},p_{2}}\left( p_{1},k\right) ,
\end{eqnarray*}%
as desired (in the last step we used (\ref{3.9})).
\end{proof}

In particular, expression (\ref{3.11}) gives us the following infinite
family of formulas for the GSN $S_{a_{1},b_{1}}^{a_{2},b_{2},p_{2}}\left(
p_{1},k\right) $ in terms of standard Stirling numbers (where $m$ is any
non-negative integer)%
\begin{eqnarray}
&&k!S_{a_{1},b_{1}}^{a_{2},b_{2},p_{2}}\left( p_{1},k\right)  \label{3.13} \\
&=&\sum_{j_{1}=0}^{p_{1}}\sum_{j_{2}=0}^{p_{2}}\binom{p_{1}}{j_{1}}\binom{%
p_{2}}{j_{2}}a_{1}^{j_{1}}a_{2}^{j_{2}}\left( b_{1}-a_{1}m\right)
^{p_{1}-j_{1}}\left( b_{2}-a_{2}m\right) ^{p_{2}-j_{2}}\times  \notag \\
&&\times \sum_{t=0}^{m}\binom{m}{t}\left( k+t\right) !S\left(
j_{1}+j_{2},k+t\right) .  \notag
\end{eqnarray}

We will see now that, for $m>0$, there is a nicer form to write (\ref{3.13}).

\begin{lemma}
\label{Lemma3}For non-negative integers $j,k$, $0\leq k\leq j\leq p$, and
any positive integer $m$, we have the identity%
\begin{equation}
\sum_{t=0}^{m}\binom{m}{t}\left( k+t\right) !S\left( j,k+t\right)
=k!\sum_{t=0}^{m-1}\left( -1\right) ^{t}s\left( m,m-t\right) S\left(
j+m-t,k+m\right) ,  \label{3.14}
\end{equation}%
where $s\left( \cdot ,\cdot \right) $ are the Stirling numbers of the first
kind.
\end{lemma}

\begin{proof}
We will use the recurrence for the Stirling numbers of the first kind $%
s\left( p,k\right) =s\left( p-1,k-1\right) +\left( p-1\right) s\left(
p-1,k\right) $, and the recurrence for the Stirling numbers of the second
kind $S\left( p,k\right) =S\left( p-1,k-1\right) +kS\left( p-1,k\right) $.
We proceed by induction on $m$: for $m=1$ one can see easily that both sides
of (\ref{3.14}) are equal to $k!S\left( j+1,k+1\right) $. Let us assume that
(\ref{3.14}) is true for a given $m\in \mathbb{N}$. Thus, we begin our
argument with $\sum_{t=0}^{m+1}\binom{m+1}{t}\left( k+t\right) !S\left(
j,k+t\right) $. First write $\binom{m+1}{t}=\binom{m}{t}+\binom{m}{t-1}$.
After some easy algebraic work, we use the induction hypothesis to get%
\begin{eqnarray*}
&&\sum_{t=0}^{m+1}\binom{m+1}{t}\left( k+t\right) !S\left( j,k+t\right) \\
&=&\sum_{t=0}^{m}\binom{m}{t}\left( k+t\right) !S\left( j,k+t\right)
+\sum_{t=0}^{m}\binom{m}{t}\left( k+t+1\right) !S\left( j,k+t+1\right) \\
&=&k!\sum_{t=0}^{m-1}\left( -1\right) ^{t}s\left( m,m-t\right) \left(
S\left( j+m-t,k+m\right) +\left( k+1\right) S\left( j+m-t,k+1+m\right)
\right) .
\end{eqnarray*}%
Now use the recurrence for the Stirling numbers of the second kind, then
again some elementary algebraic steps, and then use the recurrence for the
Stirling numbers of the first kind, to get%
\begin{eqnarray*}
&&\sum_{t=0}^{m+1}\binom{m+1}{t}\left( k+t\right) !S\left( j,k+t\right) \\
&=&k!\sum_{t=0}^{m-1}\left( -1\right) ^{t}s\left( m,m-t\right) \left(
S\left( j+m+1-t,k+m+1\right) -mS\left( j+m-t,k+m+1\right) \right) \\
&=&k!\sum_{t=0}^{m-1}\left( -1\right) ^{t}s\left( m,m-t\right) S\left(
j+m+1-t,k+m+1\right) \\
&&-k!\sum_{t=1}^{m}\left( -1\right) ^{t-1}ms\left( m,m+1-t\right) S\left(
j+m+1-t,k+m+1\right) \\
&=&k!\sum_{t=0}^{m}\left( -1\right) ^{t}\left( s\left( m,m-t\right)
+ms\left( m,m+1-t\right) \right) S\left( j+m+1-t,k+m+1\right) \\
&=&k!\sum_{t=0}^{m}\left( -1\right) ^{t}s\left( m+1,m+1-t\right) S\left(
j+m+1-t,k+m+1\right) ,
\end{eqnarray*}%
which is the desired conclusion.
\end{proof}

\begin{corollary}
\label{Cor1}The GSN $S_{a_{1},b_{1}}^{a_{2},b_{2},p_{2}}\left(
p_{1},k\right) $ can be written in terms of standard Stirling numbers as%
\begin{eqnarray}
S_{a_{1},b_{1}}^{a_{2},b_{2},p_{2}}\left( p_{1},k\right)
&=&\sum_{j_{1}=0}^{p_{1}}\sum_{j_{2}=0}^{p_{2}}\binom{p_{1}}{j_{1}}\binom{%
p_{2}}{j_{2}}a_{1}^{j_{1}}a_{2}^{j_{2}}\left( b_{1}-a_{1}m\right)
^{p_{1}-j_{1}}\left( b_{2}-a_{2}m\right) ^{p_{2}-j_{2}}\times  \notag \\
&&\times \sum_{t=0}^{m-1}\left( -1\right) ^{t}s\left( m,m-t\right) S\left(
j_{1}+j_{2}+m-t,k+m\right) ,  \label{3.15}
\end{eqnarray}%
where $m$ is any positive integer.
\end{corollary}

\begin{proof}
Formula (\ref{3.15}) comes directly from (\ref{3.13}) and (\ref{3.14}).
\end{proof}

In particular, if $b_{1}$ is a positive integer, we have from (\ref{3.15})
with $p_{2}=0,a_{1}=1$ and $m=b_{1}$ that%
\begin{equation}
S_{1,m}\left( p,k\right) =\sum_{t=0}^{m-1}\left( -1\right) ^{t}s\left(
m,m-t\right) S\left( p+m-t,k+m\right) .  \label{3.151}
\end{equation}

For $m=1,2$, formula (\ref{3.151}) gives (\ref{3.3}). For $m=3,4$ we have 
\begin{eqnarray}
S_{1,3}\left( p,k\right) &=&S\left( p+3,k+3\right) -3S\left( p+2,k+3\right)
+2S\left( p+1,k+3\right) ,  \label{3.152} \\
S_{1,4}\left( p,k\right) &=&S\left( p+4,k+4\right) -6S\left( p+3,k+4\right)
+11S\left( p+2,k+4\right) -6S\left( p+1,k+4\right) .  \notag
\end{eqnarray}

An interesting consequence of (\ref{3.151}) is the following.

\begin{corollary}
We have the following recurrence for Stirling numbers 
\begin{equation}
S\left( p_{1}+p_{2},l\right) =\sum_{k=1}^{p_{2}-1}\left( -1\right)
^{p_{2}+1+k}s\left( p_{2},k\right) S\left( p_{1}+k,l\right)
+\sum_{j=0}^{p_{1}}\binom{p_{1}}{j}p_{2}^{p_{1}-j}S\left( j,l-p_{2}\right) .
\label{3.153}
\end{equation}
\end{corollary}

\begin{proof}
From (\ref{3.71}) and (\ref{3.151}) we have that%
\begin{equation}
S_{1,p_{2}}\left( p_{1},l\right) =\sum_{j=0}^{p_{1}}\binom{p_{1}}{j}%
p_{2}^{p_{1}-j}S\left( j,l\right) =\sum_{t=1}^{p_{2}}\left( -1\right)
^{p_{2}+t}s\left( p_{2},t\right) S\left( p_{1}+t,l+p_{2}\right) .
\label{3.154}
\end{equation}%
Substitute $l$ by $l-p_{2}$ to obtain from (\ref{3.154}) the desired
conclusion (\ref{3.153}).
\end{proof}

For example, if $p_{2}=2,3$, we have%
\begin{eqnarray}
S\left( p_{1}+2,l\right) &=&S\left( p_{1}+1,l\right) +\sum_{j=0}^{p_{1}}%
\binom{p_{1}}{j}2^{p_{1}-j}S\left( j,l-2\right)  \notag \\
S\left( p_{1}+3,l\right) &=&-2S\left( p_{1}+1,l\right) +3S\left(
p_{1}+2,l\right) +\sum_{j=0}^{p_{1}}\binom{p_{1}}{j}3^{p_{1}-j}S\left(
j,l-3\right) .  \notag
\end{eqnarray}

For $p_{1}$ given and $l=p_{2}$, (\ref{3.153}) gives us formulas expressing
Stirling numbers of the second kind in terms of Stirling numbers of the
first kind. For example, for $p_{1}=1,2$ we have%
\begin{equation*}
S\left( p_{2}+1,p_{2}\right) =s\left( p_{2},p_{2}-1\right) +p_{2}.
\end{equation*}%
\begin{equation*}
S\left( p_{2}+2,p_{2}\right) =\left( s\left( p_{2},p_{2}-1\right) \right)
^{2}+p_{2}s\left( p_{2},p_{2}-1\right) -s\left( p_{2},p_{2}-2\right)
+p_{2}^{2}.
\end{equation*}

The main ideas involved in (\ref{3.7}), (\ref{3.13}) and (\ref{3.15}) can be
summarized as follows: we can write the GSN $%
S_{a_{1},b_{1}}^{a_{2},b_{2},p_{2}}\left( p_{1},k\right) $ in terms of
standard Stirling numbers by using (\ref{3.7}) (which comes from (\ref{3.13}%
) with $m=0$), but also for $m>0$ we can use (\ref{3.15}). For example, with 
$m=0,1,2$ we have 
\begin{eqnarray}
&&S_{a_{1},b_{1}}^{a_{2},b_{2},p_{2}}\left( p_{1},k\right)  \label{3.16} \\
&=&\sum_{j_{1}=0}^{p_{1}}\sum_{j_{2}=0}^{p_{2}}\binom{p_{1}}{j_{1}}\binom{%
p_{2}}{j_{2}}%
a_{1}^{j_{1}}a_{2}^{j_{2}}b_{1}^{p_{1}-j_{1}}b_{2}^{p_{2}-j_{2}}S\left(
j_{1}+j_{2},k\right)  \notag \\
&=&\sum_{j_{1}=0}^{p_{1}}\sum_{j_{2}=0}^{p_{2}}\binom{p_{1}}{j_{1}}\binom{%
p_{2}}{j_{2}}a_{1}^{j_{1}}a_{2}^{j_{2}}\left( b_{1}-a_{1}\right)
^{p_{1}-j_{1}}\left( b_{2}-a_{2}\right) ^{p_{2}-j_{2}}S\left(
j_{1}+j_{2}+1,k+1\right)  \notag \\
&=&\sum_{j_{1}=0}^{p_{1}}\sum_{j_{2}=0}^{p_{2}}\binom{p_{1}}{j_{1}}\binom{%
p_{2}}{j_{2}}a_{1}^{j_{1}}a_{2}^{j_{2}}\left( b_{1}-2a_{1}\right)
^{p_{1}-j_{1}}\left( b_{2}-2a_{2}\right) ^{p_{2}-j_{2}}\times  \notag \\
&&\times \left( S\left( j_{1}+j_{2}+2,k+2\right) -S\left(
j_{1}+j_{2}+1,k+2\right) \right) .  \notag
\end{eqnarray}

Some concrete examples of (\ref{3.16}) are the following,

\begin{eqnarray}
&&S_{1,1}^{1,0,p_{2}}\left( p_{1},k\right)  \label{3.17} \\
&=&\sum_{j_{1}=0}^{p_{1}}\binom{p_{1}}{j_{1}}S\left( j_{1}+p_{2},k\right) 
\notag \\
&=&\sum_{j_{2}=0}^{p_{2}}\binom{p_{2}}{j_{2}}\left( -1\right)
^{p_{2}-j_{2}}S\left( p_{1}+j_{2}+1,k+1\right)  \notag \\
&=&\sum_{j_{1}=0}^{p_{1}}\sum_{j_{2}=0}^{p_{2}}\binom{p_{1}}{j_{1}}\binom{%
p_{2}}{j_{2}}\left( -1\right) ^{p_{1}-j_{1}}\left( -2\right)
^{p_{2}-j_{2}}\left( S\left( j_{1}+j_{2}+2,k+2\right) -S\left(
j_{1}+j_{2}+1,k+2\right) \right) .  \notag
\end{eqnarray}

\begin{eqnarray}
&&S_{1,2}^{1,0,p_{2}}\left( p_{1},k\right)  \label{3.18} \\
&=&\sum_{j_{1}=0}^{p_{1}}\binom{p_{1}}{j_{1}}2^{p_{1}-j_{1}}S\left(
j_{1}+p_{2},k\right)  \notag \\
&=&\sum_{j_{1}=0}^{p_{1}}\sum_{j_{2}=0}^{p_{2}}\binom{p_{1}}{j_{1}}\binom{%
p_{2}}{j_{2}}\left( -1\right) ^{p_{2}-j_{2}}S\left( j_{1}+j_{2}+1,k+1\right)
\notag \\
&=&\sum_{j_{2}=0}^{p_{2}}\binom{p_{2}}{j_{2}}\left( -2\right)
^{p_{2}-j_{2}}\left( S\left( p_{1}+j_{2}+2,k+2\right) -S\left(
p_{1}+j_{2}+1,k+2\right) \right) .  \notag
\end{eqnarray}

\begin{eqnarray}
&&S_{1,1}^{1,2,p_{2}}\left( p_{1},k\right)  \label{3.19} \\
&=&\sum_{j_{1}=0}^{p_{1}}\sum_{j_{2}=0}^{p_{2}}\binom{p_{1}}{j_{1}}\binom{%
p_{2}}{j_{2}}2^{p_{2}-j_{2}}S\left( j_{1}+j_{2},k\right)  \notag \\
&=&\sum_{j_{2}=0}^{p_{2}}\binom{p_{2}}{j_{2}}S\left( p_{1}+j_{2}+1,k+1\right)
\notag \\
&=&\sum_{j_{1}=0}^{p_{1}}\binom{p_{1}}{j_{1}}\left( -1\right)
^{p_{1}-j_{1}}\left( S\left( j_{1}+p_{2}+2,k+2\right) -S\left(
j_{1}+p_{2}+1,k+2\right) \right) .  \notag
\end{eqnarray}

(In passing: from (\ref{3.19}) we see that $S_{1,1}^{1,2,p}\left( 1,k\right)
=\sum_{j=0}^{p}\binom{p}{j}S\left( j+2,k+1\right) $, and from (\ref{3.17})
we see that $S_{1,1}^{1,0,2}\left( p,k\right) =\sum_{j=0}^{p}\binom{p}{j}%
S\left( j+2,k\right) $. So we have $S_{1,2}^{1,1,1}\left( p,k\right)
=S_{1,1}^{1,0,2}\left( p,k+1\right) $ (see GSNT2 and GSNT3).)

Formula (\ref{3.6}) gives us the following relations involving the GSN (\ref%
{3.17}), (\ref{3.18}), (\ref{3.19}):%
\begin{eqnarray}
S_{1,1}^{1,0,p_{2}}\left( p_{1},k\right) &=&\sum_{j_{1}=0}^{p_{1}}\binom{%
p_{1}}{j_{1}}\left( -1\right) ^{p_{1}-j_{1}}S_{1,2}^{1,0,p_{2}}\left(
j_{1},k\right)  \label{3.20} \\
&=&\sum_{j_{2}=0}^{p_{2}}\binom{p_{2}}{j_{2}}\left( -2\right)
^{p_{2}-j_{2}}S_{1,1}^{1,2,p_{2}}\left( p_{1},k\right) ,  \notag \\
S_{1,2}^{1,0,p_{2}}\left( p_{1},k\right)
&=&\sum_{j_{1}=0}^{p_{1}}\sum_{j_{2}=0}^{p_{2}}\binom{p_{1}}{j_{1}}\binom{%
p_{2}}{j_{2}}\left( -2\right) ^{p_{2}-j_{2}}S_{1,1}^{1,2,j_{2}}\left(
j_{1},k\right) ,  \notag
\end{eqnarray}%
which, in terms of standard Stirling numbers, can be written as%
\begin{eqnarray}
\sum_{j=0}^{p}\binom{p}{j}S\left( q+j,k\right)
&=&\sum_{l=0}^{p}\sum_{j=0}^{l}\binom{p}{l}\binom{l}{j}\left( -1\right)
^{p-l}2^{l-j}S\left( q+j,k\right)  \label{3.21} \\
&=&\sum_{l=0}^{q}\sum_{j=0}^{l}\binom{q}{l}\binom{l}{j}\left( -2\right)
^{q-l}S\left( p+j+1,k+1\right) ,  \notag \\
\sum_{j=0}^{p}\binom{p}{j}2^{p-j}S\left( q+j,k\right)
&=&\sum_{j=0}^{p}\sum_{i=0}^{q}\sum_{l=0}^{i}\binom{p}{j}\binom{q}{i}\binom{i%
}{l}\left( -2\right) ^{q-i}S\left( j+l+1,k+1\right) .  \notag
\end{eqnarray}

(Of course, (\ref{3.21}) is just one of the infinite family of possibilities
to write (\ref{3.20}) in terms of standard Stirling numbers.)

Now we begin to explore a new relation involving GSN.

\begin{lemma}
\label{Lemma4}For $0\leq l\leq p_{2}+q_{1}+q_{2}$ we have%
\begin{equation}
S_{a_{1},b_{1}}^{a_{2},b_{2},q_{1}+q_{2}}\left( p_{2},l\right)
=\sum_{m=0}^{p_{2}+q_{2}}S_{a_{1},b_{1}}^{a_{2},b_{2},q_{2}}\left(
p_{2},m\right) S_{a_{2},a_{2}m+b_{2}}\left( q_{1},l-m\right) .  \label{3.22}
\end{equation}
\end{lemma}

\begin{proof}
We proceed by induction on $q_{1}$. If $q_{1}=0$ formula (\ref{3.22}) is
trivial. If we suppose it is true for $q_{1}\in \mathbb{N}$, we have%
\begin{eqnarray*}
&&\sum_{m=0}^{p_{2}+q_{2}}S_{a_{1},b_{1}}^{a_{2},b_{2},q_{2}}\left(
p_{2},m\right) S_{a_{2},a_{2}m+b_{2}}\left( q_{1}+1,l-m\right) \\
&=&\sum_{m=0}^{p_{2}+q_{2}}S_{a_{1},b_{1}}^{a_{2},b_{2},q_{2}}\left(
p_{2},m\right) \left( 
\begin{array}{c}
a_{2}S_{a_{2},a_{2}m+b_{2}}\left( q_{1},l-m-1\right) \\ 
+\left( a_{2}\left( l-m\right) +a_{2}m+b_{2}\right)
S_{a_{2},a_{2}m+b_{2}}\left( q_{1},l-m\right)%
\end{array}%
\right) \\
&=&a_{2}\sum_{m=0}^{p_{2}+q_{2}}S_{a_{1},b_{1}}^{a_{2},b_{2},q_{2}}\left(
p_{2},m\right) S_{a_{2},a_{2}m+b_{2}}\left( q_{1},l-m-1\right) \\
&&+\left( a_{2}l+b_{2}\right)
\sum_{m=0}^{p_{2}+q_{2}}S_{a_{1},b_{1}}^{a_{2},b_{2},q_{2}}\left(
p_{2},m\right) S_{a_{2},a_{2}m+b_{2}}\left( q_{1},l-m\right) \\
&=&a_{2}S_{a_{1},b_{1}}^{a_{2},b_{2},q_{1}+q_{2}}\left( p_{2},l-1\right)
+\left( a_{2}l+b_{2}\right) S_{a_{1},b_{1}}^{a_{2},b_{2},q_{1}+q_{2}}\left(
p_{2},l\right) \\
&=&S_{a_{1},b_{1}}^{a_{2},b_{2},q_{1}+q_{2}+1}\left( p_{2},l\right) ,
\end{eqnarray*}%
as desired.
\end{proof}

\begin{proposition}
\label{Prop2}For $0\leq l\leq p_{1}+p_{2}+q_{1}+q_{2}$ we have%
\begin{equation}
S_{a_{1},b_{1}}^{a_{2},b_{2},q_{1}+q_{2}}\left( p_{1}+p_{2},l\right)
=\sum_{m=0}^{p_{2}+q_{2}}S_{a_{1},b_{1}}^{a_{2},b_{2},q_{2}}\left(
p_{2},m\right) S_{a_{1},a_{1}m+b_{1}}^{a_{2},a_{2}m+b_{2},q_{1}}\left(
p_{1},l-m\right) .  \label{3.23}
\end{equation}
\end{proposition}

\begin{proof}
We proceed by induction on $p_{1}$. If $p_{1}=0$, formula (\ref{3.23}) is (%
\ref{3.22}). If formula (\ref{3.23}) is true for $p_{1}\in \mathbb{N}$, then
(by using the recurrence (\ref{3.5})) 
\begin{eqnarray*}
&&\sum_{m=0}^{p_{2}+q_{2}}S_{a_{1},b_{1}}^{a_{2},b_{2},q_{2}}\left(
p_{2},m\right) S_{a_{1},a_{1}m+b_{1}}^{a_{2},a_{2}m+b_{2},q_{1}}\left(
p_{1}+1,l-m\right) \\
&=&\sum_{m=0}^{p_{2}+q_{2}}S_{a_{1},b_{1}}^{a_{2},b_{2},q_{2}}\left(
p_{2},m\right) \left( 
\begin{array}{c}
a_{1}S_{a_{1},a_{1}m+b_{1}}^{a_{2},a_{2}m+b_{2},q_{1}}\left(
p_{1},l-m-1\right) \\ 
+\left( a_{1}\left( l-m\right) +a_{1}m+b_{1}\right)
S_{a_{1},a_{1}m+b_{1}}^{a_{2},a_{2}m+b_{2},q_{1}}\left( p_{1},l-m\right)%
\end{array}%
\right) \\
&=&\sum_{m=0}^{p_{2}+q_{2}}S_{a_{1},b_{1}}^{a_{2},b_{2},q_{2}}\left(
p_{2},m\right) \left( 
\begin{array}{c}
a_{1}S_{a_{1},a_{1}m+b_{1}}^{a_{2},a_{2}m+b_{2},q_{1}}\left(
p_{1},l-m-1\right) \\ 
+\left( a_{1}l+b_{1}\right)
S_{a_{1},a_{1}m+b_{1}}^{a_{2},a_{2}m+b_{2},q_{1}}\left( p_{1},l-m\right)%
\end{array}%
\right) \\
&=&\sum_{m=0}^{p_{2}+q_{2}}S_{a_{1},b_{1}}^{a_{2},b_{2},q_{2}}\left(
p_{2},m\right) a_{1}S_{a_{1},a_{1}m+b_{1}}^{a_{2},a_{2}m+b_{2},q_{1}}\left(
p_{1},l-m-1\right) \\
&&+\left( a_{1}l+b_{1}\right)
\sum_{m=0}^{p_{2}}S_{a_{1},b_{1}}^{a_{2},b_{2},q_{2}}\left( p_{2},m\right)
S_{a_{1},a_{1}m+b_{1}}^{a_{2},a_{2}m+b_{2},q_{1}}\left( p_{1},l-m\right) \\
&=&S_{a_{1},b_{1}}^{a_{2},b_{2},q_{1}+q_{2}}\left( p_{1}+p_{2},l-1\right)
+\left( a_{1}l+b_{1}\right) S_{a_{1},b_{1}}^{a_{2},b_{2},q_{1}+q_{2}}\left(
p_{1}+p_{2},l\right) \\
&=&S_{a_{1},b_{1}}^{a_{2},b_{2},q_{1}+q_{2}}\left( p_{1}+p_{2}+1,l\right) ,
\end{eqnarray*}%
as desired.
\end{proof}

The case $q_{1}=q_{2}=0$ of (\ref{3.23}) is%
\begin{equation}
S_{a_{1},b_{1}}\left( p_{1}+p_{2},l\right)
=\sum_{m=0}^{p_{2}}S_{a_{1},b_{1}}\left( p_{2},m\right)
S_{a_{1},a_{1}m+b_{1}}\left( p_{1},l-m\right) ,  \label{3.26}
\end{equation}%
and the standard case $a_{1}=1,b_{1}=0$ of (\ref{3.26}) can be written (by
using (\ref{3.71})) as%
\begin{equation}
S\left( p_{1}+p_{2},l\right) =\sum_{m=0}^{p_{2}}\sum_{j=0}^{p_{1}}\binom{%
p_{1}}{j}m^{p_{1}-j}S\left( j,l-m\right) S\left( p_{2},m\right) .
\label{3.27}
\end{equation}

Comparing with (\ref{3.153}), we conclude the identity (for positive $p_{2}$%
) 
\begin{equation}
\sum_{m=0}^{p_{2}-1}\sum_{j=0}^{p_{1}}\binom{p_{1}}{j}m^{p_{1}-j}S\left(
j,l-m\right) S\left( p_{2},m\right) =\sum_{k=1}^{p_{2}-1}\left( -1\right)
^{p_{2}+k+1}s\left( p_{2},k\right) S\left( p_{1}+k,l\right) .  \label{3.281}
\end{equation}

(It is possible to prove (\ref{3.281}) by induction on the non-negative
integer $p_{1}$, by using only the recurrence for standard Stirling numbers
of the second kind. It is a nice exercise left to the reader.)

It is interesting to note that, for $p_{1}$ given, formula (\ref{3.27})
gives us an explicit formula for the Stirling number $S\left(
p+p_{1},l\right) $ in terms of the Stirling numbers $S\left( p,l\right)
,S\left( p,l-1\right) ,\ldots ,S\left( p,l-p_{1}\right) $ (the case $p_{1}=0$
is trivial). If $p_{1}=1$, formula (\ref{3.27}) is just the known standard
recurrence $S\left( p+1,l\right) =lS\left( p,l\right) +S\left( p,l-1\right) $%
. For example, if $p_{1}=2,3$, formula (\ref{3.281}) gives us%
\begin{eqnarray}
S\left( p+2,l\right) &=&l^{2}S\left( p,l\right) +\left( 2l-1\right) S\left(
p,l-1\right) +S\left( p,l-2\right) ,  \label{3.28} \\
S\left( p+3,l\right) &=&l^{3}S\left( p,l\right) +\left( 3l^{2}-3l+1\right)
S\left( p,l-1\right) +3\left( l-1\right) S\left( p,l-2\right) +S\left(
p,l-3\right) ,  \notag
\end{eqnarray}%
respectively (iterations of the standard recurrence).

A final comment about (\ref{3.26}): we can use (\ref{3.15}) to write (\ref%
{3.26}) as%
\begin{eqnarray}
&&S_{a_{1},b_{1}}\left( p_{1}+p_{2},l\right) =  \label{3.30} \\
&&\sum_{m=0}^{p_{2}}\sum_{j_{1}=0}^{p_{1}}\binom{p_{1}}{j_{1}}%
a_{1}^{j_{1}}\left( b_{1}-a_{1}\left( n-m\right) \right)
^{p_{1}-j_{1}}\sum_{t=0}^{n-1}\left( -1\right) ^{t}s\left( n,n-t\right)
S\left( j_{1}+n-t,l-m+n\right) S_{a_{1},b_{1}}\left( p_{2},m\right) ,  \notag
\end{eqnarray}%
where $n$ is an arbitrary positive integer. Some particular cases of (\ref%
{3.30}) are the following (relatives of (\ref{3.27})) 
\begin{eqnarray}
S\left( p_{1}+p_{2},l\right) &=&\sum_{m=0}^{p_{2}}\sum_{j_{1}=0}^{p_{1}}%
\binom{p_{1}}{j_{1}}\left( m-1\right) ^{p_{1}-j_{1}}S\left(
j_{1}+1,l-m+1\right) S\left( p_{2},m\right)  \notag \\
&=&\sum_{m=0}^{p_{2}}\sum_{j_{1}=0}^{p_{1}}\binom{p_{1}}{j_{1}}\left(
m-2\right) ^{p_{1}-j_{1}}\left( S\left( j_{1}+2,l-m+2\right) -S\left(
j_{1}+1,l-m+2\right) \right) S\left( p_{2},m\right) .  \notag
\end{eqnarray}

\begin{eqnarray}
&&S\left( p_{1}+p_{2}+1,l+1\right)  \notag \\
&=&\sum_{m=0}^{p_{2}}\sum_{j_{1}=0}^{p_{1}}\binom{p_{1}}{j_{1}}%
m^{p_{1}-j_{1}}S\left( j_{1}+1,l-m+1\right) S\left( p_{2}+1,m+1\right) 
\notag \\
&=&\sum_{m=0}^{p_{2}}\sum_{j_{1}=0}^{p_{1}}\binom{p_{1}}{j_{1}}\left(
m-1\right) ^{p_{1}-j_{1}}\left( S\left( j_{1}+2,l-m+2\right) -S\left(
j_{1}+1,l-m+2\right) \right) S\left( p_{2}+1,m+1\right) .  \notag
\end{eqnarray}

Formula (\ref{3.23}) can be seen as an identity of two polynomials in the
variables $a_{1},b_{1},a_{2},b_{2}$. This fact produces some natural
corollaries. We show next one of them.

\begin{corollary}
\label{Cor3}For $0\leq l,t\leq p_{1}+p_{2}+q_{1}+q_{2}$, we have%
\begin{eqnarray}
&&\sum_{r=0}^{q_{1}+q_{2}}\binom{p_{1}+p_{2}+r}{t}\binom{q_{1}+q_{2}}{r}%
S\left( p_{1}+p_{2}+r-t,l\right)  \label{3.34} \\
&=&\sum_{m=0}^{p_{2}+q_{2}}\sum_{r_{1}=0}^{q_{1}}\sum_{r_{2}=0}^{q_{2}}%
\sum_{k=0}^{p_{1}+r_{1}}\sum_{s=0}^{k}\binom{q_{1}}{r_{1}}\binom{q_{2}}{r_{2}%
}\binom{p_{1}+r_{1}}{k}\binom{p_{2}+r_{2}}{t-s}\binom{k}{s}\times  \notag \\
&&\times m^{k-s}S\left( p_{2}+r_{2}-t+s,m\right) S\left(
p_{1}+r_{1}-k,l-m\right) .  \notag
\end{eqnarray}
\end{corollary}

\begin{proof}
We consider the case $a_{1}=a_{2}(=a)$, $b_{2}=b_{1}+1(=b+1)$ of (\ref{3.23}%
), that is%
\begin{equation}
\sum_{m=0}^{p_{2}+q_{2}}S_{a,b}^{a,b+1,q_{2}}\left( p_{2},m\right)
S_{a,am+b}^{a,am+b+1,q_{1}}\left( p_{1},l-m\right)
=S_{a,b}^{a,b+1,q_{1}+q_{2}}\left( p_{1}+p_{2},l\right) .  \label{3.35}
\end{equation}%
We claim that 
\begin{equation}
S_{a,b}^{a,b+1,q}\left( p,k\right) =\sum_{r=0}^{q}\binom{q}{r}S_{a,b}\left(
p+r,k\right) .  \label{3.36}
\end{equation}

In fact, we have%
\begin{eqnarray*}
S_{a,b}^{a,b+1,q}\left( p,k\right) &=&\frac{1}{k!}\sum_{j=0}^{k}\left(
-1\right) ^{j}\binom{k}{j}\left( a\left( k-j\right) +b\right) ^{p}\left(
a\left( k-j\right) +b+1\right) ^{q} \\
&=&\frac{1}{k!}\sum_{j=0}^{k}\left( -1\right) ^{j}\binom{k}{j}\left( a\left(
k-j\right) +b\right) ^{p}\sum_{r=0}^{q}\binom{q}{r}\left( a\left( k-j\right)
+b\right) ^{r} \\
&=&\sum_{r=0}^{q}\binom{q}{r}\frac{1}{k!}\sum_{j=0}^{k}\left( -1\right) ^{j}%
\binom{k}{j}\left( a\left( k-j\right) +b\right) ^{p+r} \\
&=&\sum_{r=0}^{q}\binom{q}{r}S_{a,b}\left( p+r,k\right) ,
\end{eqnarray*}%
which proves our claim. Thus, by using (\ref{3.36}) we can write (\ref{3.35}%
) as%
\begin{eqnarray}
&&\sum_{m=0}^{p_{2}+q_{2}}\sum_{r_{1}=0}^{q_{1}}\sum_{r_{2}=0}^{q_{2}}\binom{%
q_{1}}{r_{1}}\binom{q_{2}}{r_{2}}S_{a,b}\left( p_{2}+r_{2},m\right)
S_{a,am+b}\left( p_{1}+r_{1},l-m\right)  \notag \\
&=&\sum_{r=0}^{q_{1}+q_{2}}\binom{q_{1}+q_{2}}{r}S_{a,b}\left(
p_{1}+p_{2}+r,l\right) .  \label{3.37}
\end{eqnarray}

Set $a=1$ and use (\ref{3.71}) to obtain from (\ref{3.37}) that 
\begin{eqnarray}
&&\sum_{t=0}^{p_{1}+p_{2}+q_{1}+q_{2}}\sum_{r=0}^{q_{1}+q_{2}}\binom{%
p_{1}+p_{2}+r}{t}\binom{q_{1}+q_{2}}{r}b^{t}S\left( p_{1}+p_{2}+r-t,l\right)
\label{3.38} \\
&=&\sum_{m=0}^{p_{2}+q_{2}}\sum_{r_{1}=0}^{q_{1}}\sum_{r_{2}=0}^{q_{2}}%
\binom{q_{1}}{r_{1}}\binom{q_{2}}{r_{2}}\sum_{k=0}^{p_{1}+r_{1}}\binom{%
p_{1}+r_{1}}{k}\times  \notag \\
&&\times \sum_{j=0}^{p_{2}+r_{2}}\binom{p_{2}+r_{2}}{j}S\left(
p_{2}+r_{2}-j,m\right) \sum_{s=0}^{k}\binom{k}{s}b^{s+j}m^{k-s}S\left(
p_{1}+r_{1}-k,l-m\right) .  \notag
\end{eqnarray}

In the right-hand side of (\ref{3.38}) introduce the new summation index $%
t=s+j$, to obtain%
\begin{eqnarray*}
&&\sum_{t=0}^{p_{1}+p_{2}+q_{1}+q_{2}}\sum_{r=0}^{q_{1}+q_{2}}\binom{%
p_{1}+p_{2}+r}{t}\binom{q_{1}+q_{2}}{r}b^{t}S\left( p_{1}+p_{2}+r-t,l\right)
\\
&=&\sum_{m=0}^{p_{2}+q_{2}}\sum_{r_{1}=0}^{q_{1}}\sum_{r_{2}=0}^{q_{2}}%
\binom{q_{1}}{r_{1}}\binom{q_{2}}{r_{2}}\sum_{k=0}^{p_{1}+r_{1}}\binom{%
p_{1}+r_{1}}{k}\sum_{s=0}^{k}\binom{k}{s}m^{k-s}\times \\
&&\times \sum_{t=s}^{p_{2}+r_{2}+k}\binom{p_{2}+r_{2}}{t-s}S\left(
p_{2}+r_{2}-t+s,m\right) b^{t}S\left( p_{1}+r_{1}-k,l-m\right) ,
\end{eqnarray*}

from where we obtain the desired conclusion (\ref{3.34}).
\end{proof}

The case $q_{1}=q_{2}=0$ of (\ref{3.34}) says that for $0\leq l,t\leq
p_{1}+p_{2}$ we have%
\begin{eqnarray}
&&\binom{p_{1}+p_{2}}{t}S\left( p_{1}+p_{2}-t,l\right)  \label{3.39} \\
&=&\sum_{m=0}^{p_{2}}\sum_{k=0}^{p_{1}}\sum_{s=0}^{k}\binom{p_{1}}{k}\binom{%
p_{2}}{t-s}\binom{k}{s}m^{k-s}S\left( p_{2}-t+s,m\right) S\left(
p_{1}-k,l-m\right) ,  \notag
\end{eqnarray}%
(the case $t=0$ of (\ref{3.39}) is (\ref{3.27}).

When $t=p_{1}+p_{2}$ expression (\ref{3.34}) looks as%
\begin{eqnarray}
\sum_{r=0}^{q_{1}+q_{2}}\!\binom{p_{1}+p_{2}+r}{r}\!\binom{q_{1}+q_{2}}{r}%
\,\!S\left( r,l\right) \!\!\!
&=&\!\!\!\sum_{m=0}^{p_{2}+q_{2}}\sum_{r_{1}=0}^{q_{1}}%
\sum_{r_{2}=0}^{q_{2}}\sum_{k=0}^{p_{1}+r_{1}}\sum_{s=0}^{k}\binom{q_{1}}{%
r_{1}}\binom{q_{2}}{r_{2}}\binom{p_{1}+r_{1}}{k}\binom{p_{2}+r_{2}}{%
p_{1}+p_{2}-s}\binom{k}{s}\times  \notag \\
&&\times m^{k-s}S\left( r_{2}-p_{1}+s,m\right) S\left(
p_{1}+r_{1}-k,l-m\right) .  \label{3.401}
\end{eqnarray}

If we set $l=1$ in (\ref{3.401}), we obtain the combinatorial identity%
\begin{eqnarray}
\sum_{r=1}^{q_{1}+q_{2}}\binom{p_{1}+p_{2}+r}{r}\binom{q_{1}+q_{2}}{r}\!\!\!
&=&\!\!\!\sum_{r_{2}=1}^{q_{2}}\binom{q_{2}}{r_{2}}\!\left( \binom{p_{1}}{%
r_{2}}+\sum_{s=0}^{r_{2}-1}\!\binom{p_{2}+r_{2}}{r_{2}-s}\binom{p_{1}}{s}%
\right)  \label{3.41} \\
&&+\sum_{r_{1}=1}^{q_{1}}\binom{q_{1}}{r_{1}}\!\left( \binom{p_{1}+r_{1}}{%
r_{1}}+\sum_{s=0}^{r_{1}-1}\!\binom{p_{2}}{r_{1}-s}\binom{p_{1}+r_{1}}{s}%
\right)  \notag \\
&&+\sum_{r_{1}=1}^{q_{1}}\sum_{r_{2}=1}^{q_{2}}\binom{q_{1}}{r_{1}}\binom{%
q_{2}}{r_{2}}\!\left( \binom{p_{1}+r_{1}}{r_{1}+r_{2}}%
+\sum_{s=0}^{r_{1}+r_{2}-1}\!\binom{p_{2}+r_{2}}{r_{2}+r_{1}-s}\binom{%
p_{1}+r_{1}}{s}\right) .  \notag
\end{eqnarray}

Formula (\ref{3.23}) can be generalized to the case of sums of $N$
non-negative integers $p_{1}+\cdots +p_{N}$ and $q_{1}+\cdots +q_{N}$. We
just report the case $N=3$: 
\begin{eqnarray}
&&S_{a_{1},b_{1}}^{a_{2},b_{2},q_{1}+q_{2}+q_{3}}\left(
p_{1}+p_{2}+p_{3},l\right)  \label{3.42} \\
&=&\sum_{n=0}^{p_{3}+q_{3}}%
\sum_{m=0}^{p_{2}+q_{2}}S_{a_{1},b_{1}}^{a_{2},b_{2},q_{3}}\left(
p_{3},n\right) S_{a_{1},a_{1}n+b_{1}}^{a_{2},a_{2}n+b_{2},q_{2}}\left(
p_{2},m\right) S_{a_{1},a_{1}\left( m+n\right) +b_{1}}^{a_{2},a_{2}\left(
m+n\right) +b_{2},q_{1}}\left( p_{1},l-n-m\right) ,  \notag
\end{eqnarray}%
(sketch of the proof: in (\ref{3.23}) substitute $b_{1}$ by $a_{1}n+b_{1}$, $%
b_{2}$ by $a_{2}n+b_{2}$, and $l$ by $l-n$, then multiply both sides by $%
S_{a_{1},b_{1}}^{a_{2},b_{2},q_{3}}\left( p_{3},n\right) $, then take\ the
sum $\sum_{n=0}^{p_{3}+q_{3}}$ in both sides, and finally use (\ref{3.23})
to obtain (\ref{3.42})).

The following result shows that the GSN $S_{a_{1},b_{1}}^{a_{2},b_{2},p_{2}}%
\left( p_{1},t\right) $ are involved in some (weighted and/or generalized)
sums of powers formulas (as expected, see \cite{Boya,Wit}).

\begin{proposition}
For integers $m>$ $0$ and $0\leq k\leq m-1$, we have%
\begin{equation}
\sum_{t=0}^{m-k-1}\binom{m}{t+k+1}%
t!S_{a_{1},b_{1}+a_{1}m}^{a_{2},b_{2}+a_{2}m,p_{2}}\left( p_{1},t\right)
=\sum_{t=k+1}^{m}\binom{t-1}{k}\left( b_{1}+a_{1}t\right) ^{p_{1}}\left(
b_{2}+a_{2}t\right) ^{p_{2}}.  \label{3.43}
\end{equation}
\end{proposition}

\begin{proof}
We proceed by induction on $m$. If $m=1$ (and then $k=0$) we have%
\begin{equation*}
S_{a_{1},b_{1}+a_{1}}^{a_{2},b_{2}+a_{2},p_{2}}\left( p_{1},0\right) =\left(
b_{1}+a_{1}\right) ^{p_{1}}\left( b_{2}+a_{2}\right) ^{p_{2}}.
\end{equation*}

If it is true for an $m$, then%
\begin{eqnarray*}
&&\sum_{t=0}^{m-k}\binom{m+1}{t+k+1}%
t!S_{a_{1},b_{1}+a_{1}+a_{1}m}^{a_{2},b_{2}+a_{2}+a_{2}m,p_{2}}\left(
p_{1},t\right) \\
&=&\sum_{t=0}^{m-k-1}\binom{m}{t+k+1}%
t!S_{a_{1},b_{1}+a_{1}+a_{1}m}^{a_{2},b_{2}+a_{2}+a_{2}m,p_{2}}\left(
p_{1},t\right) +\sum_{t=0}^{m-k}\binom{m}{t+k}%
t!S_{a_{1},b_{1}+a_{1}+a_{1}m}^{a_{2},b_{2}+a_{2}+a_{2}m,p_{2}}\left(
p_{1},t\right) \\
&=&\sum_{t=k+1}^{m}\binom{t-1}{k}\left( b_{1}+a_{1}+a_{1}t\right)
^{p_{1}}\left( b_{2}+a_{2}+a_{2}t\right) ^{p_{2}}+\sum_{t=k}^{m}\binom{t-1}{%
k-1}\left( b_{1}+a_{1}+a_{1}t\right) ^{p_{1}}\left(
b_{2}+a_{2}+a_{2}t\right) ^{p_{2}} \\
&=&\sum_{t=k+1}^{m}\binom{t}{k}\left( b_{1}+a_{1}+a_{1}t\right)
^{p_{1}}\left( b_{2}+a_{2}+a_{2}t\right) ^{p_{2}}+\left(
b_{1}+a_{1}+a_{1}k\right) ^{p_{1}}\left( b_{2}+a_{2}+a_{2}k\right) ^{p_{2}}
\\
&=&\sum_{t=k+2}^{m+1}\binom{t-1}{k}\left( b_{1}+a_{1}t\right) ^{p_{1}}\left(
b_{2}+a_{2}t\right) ^{p_{2}}+\left( b_{1}+a_{1}+a_{1}k\right) ^{p_{1}}\left(
b_{2}+a_{2}+a_{2}k\right) ^{p_{2}} \\
&=&\sum_{t=k+1}^{m+1}\binom{t-1}{k}\left( b_{1}+a_{1}t\right) ^{p_{1}}\left(
b_{2}+a_{2}t\right) ^{p_{2}},
\end{eqnarray*}

as desired.
\end{proof}

Two examples of (\ref{3.43}) are 
\begin{equation*}
\sum_{t=k+1}^{m}\binom{t-1}{k}\left( m-t\right) \left( t-m-1\right) ^{2}=4%
\binom{m+1}{k+3}+6\binom{m+1}{k+4}.
\end{equation*}%
\begin{equation*}
\sum_{t=k+1}^{m}\binom{t-1}{k}\left( m-t\right) ^{3}=\binom{m}{k+2}+6\binom{%
m+1}{k+4}.
\end{equation*}

To end this section we show a convolution formula involving the GSN.

\begin{proposition}
For arbitrary non-negative integers $k$ and $\mu $ we have%
\begin{equation}
\binom{k+\mu }{k}S_{a_{1},b_{1}}^{a_{2},b_{2},p_{2}}\left( p_{1},k+\mu
\right) =\sum_{j_{2}=0}^{p_{2}}\sum_{j_{1}=0}^{p_{1}}\binom{p_{1}}{j_{1}}%
\binom{p_{2}}{j_{2}}S_{a_{1},0}^{a_{2},0,p_{2}-j_{2}}\left( p_{1}-j_{1},\mu
\right) S_{a_{1},b_{1}}^{a_{2},b_{2},j_{2}}\left( j_{1},k\right) .
\label{3.44}
\end{equation}
\end{proposition}

\begin{proof}
We proceed by induction on $k$. For $k=0$ we have%
\begin{eqnarray*}
&&\sum_{j_{2}=0}^{p_{2}}\sum_{j_{1}=0}^{p_{1}}\binom{p_{1}}{j_{1}}\binom{%
p_{2}}{j_{2}}S_{a_{1},0}^{a_{2},0,p_{2}-j_{2}}\left( p_{1}-j_{1},\mu \right)
S_{a_{1},b_{1}}^{a_{2},b_{2},j_{2}}\left( j_{1},0\right) \\
&=&\sum_{j_{2}=0}^{p_{2}}\sum_{j_{1}=0}^{p_{1}}\binom{p_{1}}{j_{1}}\binom{%
p_{2}}{j_{2}}b_{1}^{j_{1}}b_{2}^{j_{2}}S_{a_{1},0}^{a_{2},0,p_{2}-j_{2}}%
\left( p_{1}-j_{1},\mu \right) \\
&=&\sum_{j_{2}=0}^{p_{2}}\sum_{j_{1}=0}^{p_{1}}\binom{p_{1}}{j_{1}}\binom{%
p_{2}}{j_{2}}b_{1}^{j_{1}}b_{2}^{j_{2}}\frac{1}{\mu !}\sum_{s=0}^{\mu
}\left( -1\right) ^{s}\binom{\mu }{s}\left( a_{1}\left( \mu -s\right)
\right) ^{p_{1}-j_{1}}\left( a_{2}\left( \mu -s\right) \right) ^{p_{2}-j_{2}}
\\
&=&\frac{1}{\mu !}\sum_{s=0}^{\mu }\left( -1\right) ^{s}\binom{\mu }{s}%
\left( a_{1}\left( \mu -s\right) +b_{1}\right) ^{p_{1}}\left( a_{2}\left(
\mu -s\right) +b_{2}\right) ^{p_{2}} \\
&=&S_{a_{1},b_{1}}^{a_{2},b_{2},p_{2}}\left( p_{1},\mu \right) ,
\end{eqnarray*}

as desired. If formula (\ref{3.44}) is true for a given $k\in \mathbb{N}$,
we have (by using (\ref{3.9}))%
\begin{eqnarray*}
&&\sum_{j_{2}=0}^{p_{2}}\sum_{j_{1}=0}^{p_{1}}\binom{p_{1}}{j_{1}}\binom{%
p_{2}}{j_{2}}S_{a_{1},0}^{a_{2},0,p_{2}-j_{2}}\left( p_{1}-j_{1},\mu \right)
S_{a_{1},b_{1}}^{a_{2},b_{2},j_{2}}\left( j_{1},k+1\right) \\
&=&\frac{1}{k+1}\sum_{j_{2}=0}^{p_{2}}\sum_{j_{1}=0}^{p_{1}}\binom{p_{1}}{%
j_{1}}\binom{p_{2}}{j_{2}}S_{a_{1},0}^{a_{2},0,p_{2}-j_{2}}\left(
p_{1}-j_{1},\mu \right) \left(
S_{a_{1},a_{1}+b_{1}}^{a_{2},a_{2}+b_{2},j_{2}}\left( j_{1},k\right)
-S_{a_{1},b_{1}}^{a_{2},b_{2},j_{2}}\left( j_{1},k\right) \right) \\
&=&\frac{1}{k+1}\left( 
\begin{array}{c}
\sum_{j_{2}=0}^{p_{2}}\sum_{j_{1}=0}^{p_{1}}\binom{p_{1}}{j_{1}}\binom{p_{2}%
}{j_{2}}S_{a_{1},0}^{a_{2},0,p_{2}-j_{2}}\left( p_{1}-j_{1},\mu \right)
S_{a_{1},a_{1}+b_{1}}^{a_{2},a_{2}+b_{2},j_{2}}\left( j_{1},k\right) \\ 
-\sum_{j_{2}=0}^{p_{2}}\sum_{j_{1}=0}^{p_{1}}\binom{p_{1}}{j_{1}}\binom{p_{2}%
}{j_{2}}S_{a_{1},0}^{a_{2},0,p_{2}-j_{2}}\left( p_{1}-j_{1},\mu \right)
S_{a_{1},b_{1}}^{a_{2},b_{2},j_{2}}\left( j_{1},k\right)%
\end{array}%
\right) \\
&=&\frac{1}{k+1}\left( \binom{k+\mu }{k}%
S_{a_{1},a_{1}+b_{1}}^{a_{2},a_{2}+b_{2},p_{2}}\left( p_{1},k+\mu \right) -%
\binom{k+\mu }{k}S_{a_{1},b_{1}}^{a_{2},b_{2},p_{2}}\left( p_{1},k+\mu
\right) \right) \\
&=&\frac{1}{k+1}\binom{k+\mu }{k}\left(
S_{a_{1},a_{1}+b_{1}}^{a_{2},a_{2}+b_{2},p_{2}}\left( p_{1},k+\mu \right)
-S_{a_{1},b_{1}}^{a_{2},b_{2},p_{2}}\left( p_{1},k+\mu \right) \right) \\
&=&\frac{1}{k+1}\binom{k+\mu }{k}\left( k+\mu +1\right)
S_{a_{1},b_{1}}^{a_{2},b_{2},p_{2}}\left( p_{1},k+\mu +1\right) \\
&=&\binom{k+\mu +1}{k+1}S_{a_{1},b_{1}}^{a_{2},b_{2},p_{2}}\left(
p_{1},k+\mu +1\right) ,
\end{eqnarray*}%
as wanted.
\end{proof}

(The case $a_{1}=a_{2}=1,b_{1}=b_{2}=0$ of (\ref{3.44}) is formula (6.28),
p. 265 in \cite{G-K-P}.)

\section{\label{Sec5}Further Results: Recurrence and Differential Operator}

The GSN $S_{1,b,r}^{\alpha _{s},\beta _{s},r_{s},p_{s}}\left( p,k\right) $
satisfy the recurrence%
\begin{equation}
S_{1,b,r}^{\alpha _{s},\beta _{s},r_{s},p_{s}}\left( p,k\right)
=r!\sum_{t=0}^{r}\frac{1}{t!}\binom{b+k-t}{r-t}S_{1,b,r}^{\alpha _{s},\beta
_{s},r_{s},p_{s}}\left( p-1,k-t\right) .  \label{5.1}
\end{equation}

In terms of Differential Operators, the GSN $S_{1,b,r}\left( p,k\right) $
can be defined by%
\begin{equation}
\left( r!\sum_{j=0}^{r}\frac{\binom{b}{r-j}}{j!}x^{j}\frac{d^{j}}{dx^{j}}%
\right) ^{p}=\sum_{k=0}^{rp}S_{1,b,r}\left( p,k\right) x^{k}\frac{d^{k}}{%
dx^{k}}.  \label{5.2}
\end{equation}

Formulas (\ref{5.1}) and (\ref{5.2}) are (some of) the main results of the
second part of this work \cite{Pi3}.

\_\_\_\_\_\_\_\_\_\_\_\_\_\_\_\_\_\_\_\_\_\_\_\_\_\_\_\_\_\_\_\_\_\_\_\_\_\_%
\_\_\_\_\_\_\_\_\_\_\_\_\_\_\_\_\_\_\_\_\_\_\_\_\_\_\_\_\_\_\_\_\_\_\_\_\_\_%
\_\_\_\_\_\_\_\_\_\_\_\_\_\_\_\_\_\_\_\_\_\_\_\_\_\_\_\_\_\_\_\_\_\_\_

2010 Mathematics Subject Classification: Primary 11B73; Secondary 11B83.

Keywords: generalized Stirling numbers of the second kind, Stirling number
identities.

\end{document}